\documentclass[11pt,thmsa]{article}

\usepackage{amssymb,latexsym}
\usepackage{amsmath,amscd,amsfonts,epsf}
\usepackage[all]{xy}

 \usepackage[latin1]{inputenc}

 \usepackage{mathrsfs}
 \usepackage{amsmath}
 \usepackage{amsthm}
 \usepackage{amsfonts}
 \usepackage{amssymb}
 \usepackage[pdftex]{graphicx}
 \usepackage{wrapfig}
 \usepackage{layout}
 \usepackage{verbatim}
 \usepackage{alltt}
 \usepackage{xypic}
 \usepackage{tikz}

 \input xy
 \xyoption{all}

 \DeclareMathAlphabet{\mathpzc}{OT1}{pzc}{m}{it}

\newtheorem*{theoremnn}{Theorem~\ref{main theorem}}
\theoremstyle{definition}

 \newtheorem{theorem}{Theorem}[section]
 \newtheorem{lemma}[theorem]{Lemma}
 \newtheorem{proposition}[theorem]{Proposition}
 
 \newtheorem{corollary}[theorem]{Corollary}
 \newtheorem{definition}[theorem]{Definition}
  \theoremstyle{definition}
 \newtheorem{example}[theorem]{Example}
 \newtheorem{remark}[theorem]{Remark}

\newcommand{\id}{\mathrm{id}}
\newcommand{\End}{\textrm{End}}
\newcommand{\Hom}{\mathsf{Hom}}

\newcommand{\Path}{\mathsf{P}}

\newcommand{\Chen}{\mathsf{C}}
\newcommand{\B}{\mathsf{B}}
\newcommand{\dgcc}{\mathsf{DGCC}}
\newcommand{\dgc}{\mathsf{DGC}}
\newcommand{\dga}{\mathsf{DGA}}
\newcommand{\Linfty}{\mathsf{L}_{\infty}}
\newcommand{\Ainfty}{\mathsf{A}_{\infty}}
\newcommand{\dgla}{\mathsf{DGLA}}
\newcommand{\CE}{\mathsf{CE}}
\newcommand{\Lie}{\mathsf{L}}
\newcommand{\U}{\mathbb{U}}
\newcommand{\Uinfty}{\mathbb{U}_{\infty}}
\newcommand{\T}{\mathsf{T}}
\newcommand{\LF}{\mathsf{L}}
\newcommand{\Sym}{\mathsf{S}}
\newcommand{\ST}{\mathbb{S}}
\newcommand{\s}{\mathsf{s}}
\newcommand{\CG}{\mathsf{CG}}
\newcommand{\td}{\mathsf{t}}
\newcommand{\us}{\mathsf{u}}

\newcommand{\SW}{\mathsf{SW}}
\newcommand{\Graphs}{\mathsf{Graphs}}
\newcommand{\g}{\mathfrak{g}}
\newcommand{\Conf}{\mathsf{Conf}}
\newcommand{\FM}{\mathsf{FM}}
\newcommand{\hol}{\mathsf{hol}}
\newcommand{\holinfty}{\mathsf{hol}^{\infty}}
\newcommand{\hatUinfty}{\hat{\mathbb{U}}_{\infty}}

\def\tetra{\draw(A) -- (B) -- (C) -- (D) -- cycle; \draw[shade](A)--(B)--(D); \draw[shade](B)--(C)--(D);\draw(B)--(D); \draw[dotted](A)--(C);\draw[dotted](A)--(C); \draw(A)--(D);}
\def\tetraA{\draw(A) -- (B) -- (C) -- (D) -- cycle; \draw[shade=gray](B)--(C)--(D);\draw[thin](B)--(D); \draw[dotted](A)--(C);\draw[dotted](A)--(C); \draw[thick](A)--(B);}
\def\tetraB{\draw(A) -- (B) -- (C) -- (D) -- cycle; \shade(A)--(C)--(D);\draw(B)--(D); \draw(A)--(D); \draw[dotted](A)--(C);\draw(C)--(D);}
\def\tetraC{\draw(A) -- (B) -- (C) -- (D) -- cycle; \draw[shade=gray](A)--(B)--(D);\draw(B)--(D); \draw[dotted](A)--(C); \draw[dotted](A)--(C); \draw(A)--(D);\draw[thick](D)--(C);}
\def\tetraD{\draw(A) -- (B) -- (C) -- (D) -- cycle; \draw[shade=gray](A)--(B)--(C);\draw(B)--(D); \draw[dotted](A)--(C);\draw[dotted](A)--(C);\draw(A)--(D);  }

\def\triangle{\draw(A) -- (B) -- (C)--cycle; \draw[shade](A)--(B)--(C);\draw(A)--(C);}
\def\triangleA{\draw(A) -- (B) -- (C)  -- cycle;\draw[ultra thick](A)--(B);\draw[ultra thick](C)--(B);}
 \def\triangleB{\draw(A) -- (B) -- (C)  -- cycle;\draw[ultra thick](A)--(C);}

\def\line{\draw[ultra thick](A) -- (B);\fill[ultra thick](A);}

\def\bleft{\draw(X) -- (Y) -- (Z)--(W);}
\def\bright{\draw(x) -- (y) -- (z)--(w);}

\begin{document}

\vspace{15cm}
 \title{Holonomies for connections with values in $L_\infty$-algebras}
\author{Camilo Arias Abad\footnote{Max Planck Institut f\"ur Mathematik, Bonn.
camiloariasabad@gmail.com. Partially supported
by SNF Grant 200020-131813/1 and a Humboldt Fellowship. } \hspace{0cm} and
Florian Sch\"atz\footnote{Centre for Quantum Geometry of Moduli Spaces, Ny Munkegade 118,
DK-8000 Aarhus C, Denmark. florian.schaetz@gmail.com. Supported by ERC Starting Grant no. 279729 and the Danish National Research Foundation grant DNRF95 (Centre for Quantum Geometry of Moduli Spaces - QGM).}}

 \maketitle
 \begin{abstract}

Given a flat connection $\alpha$ on a manifold $M$ with values in a filtered $L_\infty$-algebra $\g$, we construct a
morphism $\holinfty_\alpha:C_\bullet(M)\rightarrow \B \hat{\U}_\infty(\g)$, generalizing the holonomies of flat connections with values in Lie algebras.
The construction is based on Gugenheim's $\Ainfty$-version of de Rham's theorem, which in turn is based on Chen's iterated integrals.
Finally, we discuss examples related to the geometry of configuration spaces of points in Euclidean space $\mathbb{R}^d$, and to generalizations of the holonomy representations of braid groups.
\end{abstract}

\tableofcontents

\section{Introduction}\label{s:intro}

In this note we propose an answer to the following question: Assume that
$M$ is a smooth manifold, $\g$ an $L_\infty$-algebra and $\alpha$ a flat connection on $M$ with values in $\g$, i.e.,\ a Maurer--Cartan element of the $L_\infty$-algebra
$\g \hat{\otimes} \Omega(M)$; what are the holonomies associated to the flat connection $\alpha$? 
Our answer differs from those that have appeared in the literature, such as
\cite{Picken,SatiSchreiberStasheff,SchreiberWaldorf,Tradleral,Yekutieli}, where various notions of two-dimensional parallel transport
are considered.
In order to motivate our answer, let us first discuss the case where $\g=\End V$ is the Lie algebra of endomorphisms
of a finite-dimensional vector space. In this case, $\alpha$ is just a flat
connection on the trivial vector bundle $V$, and by solving the
differential equation for parallel transport, one obtains the holonomy
$\hat{\hol}(\sigma)\in \End V\subset \U(\End V)$ associated to a path
$\sigma\colon I \rightarrow M$. One can view this whole assignment as an element $\hat{\hol}$ of $\End V \otimes C^{\bullet}(M)$,
the differential graded algebra of $\End V$-valued smooth singular cochains on $M$. The flatness of $\alpha $ implies the homotopy invariance of the holonomy. This corresponds to the
fact that $\hat{\hol}$ is a Maurer--Cartan element. Indeed, an element
$\beta \in \End V \otimes C^1(M)$ is a Maurer--Cartan element precisely if
it is homotopy invariant in the sense that for any two-dimensional simplex one has
\begin{center}
\begin{tikzpicture}[scale=1]
    \coordinate (A) at (-1,-0.5);
    \coordinate (B) at (0,-0.5);
    \coordinate (C) at (0,0.5);
    \coordinate (X) at (0,0.5);
    \coordinate (Y) at (-0.1,0.5);
    \coordinate (Z) at (-0.1,-0.5);
    \coordinate (W) at (0,-0.5);
      \coordinate (x) at (-0.5,0.5);
    \coordinate (y) at (-0.4,0.5);
    \coordinate (z) at (-0.4,-0.5);
     \coordinate (w) at (-0.5,-0.5);

     \matrix[column sep=0.8cm,row sep=0.5cm]
{
\triangleA &node[$-$]& \triangleB&   node{$=$}&&node{$0$}node[$.$]\\
};
\end{tikzpicture}
\end{center}
Here the bold edges represent holonomies associated to the corresponding paths, and concatenation of paths corresponds
to multiplication in the algebra $\End V$. Observe that a Maurer--Cartan element of $
\End V \otimes C^{\bullet}(M)$ corresponds naturally to a morphism of differential graded coalgebras
$C_\bullet(M)\rightarrow \B(\End V)$.

Using the explicit iterated integral formulas for the parallel transport, one can show that this morphism factors through the bar coalgebra of the (completed) universal enveloping algebra of $\End V$:
\begin{align*}
\xymatrix{
C_{\bullet}(M) \ar[r]^{{\hol_\alpha}\,\,} \ar[rd]&\B \hat{\U}(\End V)\ar[d]^p\\
& \B(\End V).
}
\end{align*}

This construction works for any filtered Lie algebra $\mathfrak{g}$, and we conclude that the holonomies of a flat connection with values in
$\g$ can be interpreted as a morphism of differential graded coalgebras
$\hol_{\alpha}\colon C_{\bullet}(M)\rightarrow \B \hat{\U}(\g),$ where $\B\hat{\U}(\g)$ denotes the bar construction of the completion of the universal enveloping algebra $\U(\g)$.

The case where the $L_\infty$-algebra $\g$ is the graded Lie algebra of
endomorphisms of a graded vector space $V$ corresponds to holonomies of
flat $\mathbb{Z}$-graded connections. This has been studied recently by
Igusa \cite{I}, Block and Smith \cite{BS}, and Arias Abad and Sch\"atz \cite{AS}, and ultimately relies on Gugenheim's \cite{G} $\Ainfty$-version of de Rham's theorem. In turn, Gugenheim's construction is based on Chen's theory of iterated integrals \cite{Chen}.
We extend this approach to flat connections with values in $L_\infty$-algebras. The holonomy of $\alpha$ is a morphism of differential graded coalgebras
$\hol^\infty_\alpha\colon C_\bullet(M) \rightarrow \B
\hat{\U}_\infty(\g)$.\footnote{Throughout the introduction, we gloss over the technical issue that one has to work with the completed bar complex $\hat{\B}\hatUinfty(\g)$ of $\hatUinfty(\g)$, which is not a differential graded
coalgebra, because its ``comultiplication'' does not map into the tensor product, but into the completion. See Appendix~\ref{section:twisting_cochains} for details.}

We first need to explain what the universal enveloping algebra $\Uinfty(\g)$ of an $L_\infty$-algebra $\g$ is.
Several proposals  for a definition of the enveloping algebra of an
$L_\infty$-algebra exist in the literature, e.g.,\ \cite{Rossi,B, LM}.
Following \cite{B}, we use the idea of defining the enveloping algebra via
the strictification $\mathbb{S}(\g)$ of the $L_\infty$-algebra $\g$. The
differential graded Lie algebra $\mathbb{S}(\g)$ is naturally
quasi-isomomorhic to $\g$, and we define the enveloping algebra of $\g$ to
be that of its strictification. Our main result is as follows:

\begin{theoremnn}

Suppose that $\alpha$ is a flat connection on $M$ with values in a filtered $L_\infty$-algebra  $\mathfrak{g}$. Then there is a natural homomorphism of differential graded coalgebras
\[\hol^\infty_\alpha\colon C_\bullet(M) \rightarrow \B\hat {\mathbb{U}}_\infty(\g).\]

\end{theoremnn}

In order for this notion of holonomy to be reasonable, it should be consistent with the standard definition in the
case of Lie algebras. Indeed, in the case where $\g$ is a Lie algebra, the usual parallel transport provides a
holonomy map:
\[{\hol}\colon C_\bullet(M)\rightarrow \B \hat{\U}(\g).\]
On the other hand, there is a natural map of differential graded coalgebras

\[\B \hat{\U}(\rho)\colon \B \hat{\U}_\infty (\g)\rightarrow \B \hat{\U}(\g),\]
and the following diagram commutes:

\begin{align*}
\xymatrix{
 C_\bullet(M)  \ar[r]^{\hol^\infty_\alpha} \ar[rd]_{{\hol}_\alpha}& \B \hat{\U}_\infty (\g) \ar[d]^{\B \hat{\U}(\rho)}\\
 &\B \hat{\U}(\g).
}
\end{align*}

The notion of holonomy on which Theorem~\ref{main theorem} is based admits
a rather visual description. Given any filtered differential graded algebra
$(\mathsf{A},\partial)$, a morphism of differential graded coalgebras
$\phi\colon  C_\bullet(M)\rightarrow \B \hat{A}$ corresponds to a Maurer--Cartan  element $\overline{\phi}$ in the algebra
$A \hat{\otimes} C^{\bullet}(M)$, which is an element in the vector space $\Hom(C_\bullet(M), A)$. Thus, $\phi$ can be interpreted as a rule that assigns to each simplex in $M$ an element of the algebra $\hat{A}$, which we think of as being the holonomy associated to that simplex.

Since the algebra $A\hat{\otimes} C^{\bullet}(M)$ is bigraded, the
condition for $\overline{\phi}$ to be Maurer--Cartan decomposes into a sequence of equations.
In degree~0, the condition is that $\phi$ assigns to every point $p \in
M$ a Maurer--Cartan element of $\hat{A}$.  This implies that if we set $\partial_p:=\partial+[\phi(p),\_],$ then $\partial_p \circ \partial_p=0$.
Let us denote the complex $(A,\partial_p)$ by $A_p$.
Given a simplex $\sigma: \Delta_k \to M$, we denote
the commutator between the operation of multiplying by $\phi(\sigma)$ and of applying the differentials associated
to the first and last vertex of $\sigma$ by $[\partial, \phi(\sigma)]$, i.e.,
$[\partial, \phi(\sigma)]:= \partial_{v_k} \circ \phi(\sigma)-(-1)^{1+|\sigma|} \phi(\sigma)\circ \partial_{v_0}$.

The Maurer--Cartan equation in degree $1$ is
\begin{center}
\begin{tikzpicture}[scale=1]
    \coordinate (A) at (-0.5,0);
    \coordinate (B) at (0.5,0);
    \coordinate (X) at (0,0.5);
    \coordinate (Y) at (-0.1,0.5);
    \coordinate (Z) at (-0.1,-0.5);
    \coordinate (W) at (0,-0.5);
     \coordinate (x) at (-0.5,0.5);
    \coordinate (y) at (-0.4,0.5);
    \coordinate (z) at (-0.4,-0.5);
     \coordinate (w) at (-0.5,-0.5);

\matrix[column sep=0.8cm,row sep=0.5cm]
{
\bleft &node{$\partial ,$} & \line &\bright & node{$=$}&&node{$0\quad ,$}\\
};
\end{tikzpicture}
\end{center}
\noindent
which says that multiplication by the holonomy associated to a path is an isomorphism between the complexes $A_{v_0}$ and $A_{v_1}$.
The equation in degree~2 reads

\begin{center}
\begin{tikzpicture}[scale=1]
    \coordinate (A) at (-1,-0.5);
    \coordinate (B) at (0,-0.5);
    \coordinate (C) at (0,0.5);
    \coordinate (X) at (0,0.5);
    \coordinate (Y) at (-0.1,0.5);
    \coordinate (Z) at (-0.1,-0.5);
    \coordinate (W) at (0,-0.5);
      \coordinate (x) at (-0.5,0.5);
    \coordinate (y) at (-0.4,0.5);
    \coordinate (z) at (-0.4,-0.5);
     \coordinate (w) at (-0.5,-0.5);

     \matrix[column sep=0.8cm,row sep=0.5cm]
{
\bleft & node{$\partial ,$} &\triangle&\bright & node{$=$}&\triangleA &node[$-$]& \triangleB&  node[$,$]\\
};
\end{tikzpicture}
\end{center}
\noindent
requiring that the two isomorphisms between the complexes $A_{v_0}$ and $A_{v_2}$ are homotopic, with a specified homotopy given by the holonomy associated to the triangle.

Similarly, for the tetrahedron one obtains

\begin{tikzpicture}[scale=1]
    \coordinate (A) at (-0.8,-0.3);
    \coordinate (B) at (0,-0.5);
    \coordinate (D) at (-0.1,0.5);
    \coordinate (C) at (0.5,-0.1);
     \coordinate (X) at (0,0.5);
    \coordinate (Y) at (-0.1,0.5);
    \coordinate (Z) at (-0.1,-0.5);
    \coordinate (W) at (0,-0.5);
     \coordinate (x) at (-0.5,0.5);
    \coordinate (y) at (-0.4,0.5);
    \coordinate (z) at (-0.4,-0.5);
     \coordinate (w) at (-0.5,-0.5);

\matrix[column sep=0.8cm,row sep=0.5cm]
{
\bleft & node{$\partial ,$} &\tetra &\bright  node{$=$} 
\\ && \tetraA &node[$-$]& \tetraC & node{$+$}& \tetraD& node{$-$}& \tetraB &  node[$.$]\\
};
\end{tikzpicture}

Our main motivation to develop this version of parallel transport is the
appearance of certain flat connections on configuration spaces $\Conf_d(n)$ of $n$ points in Euclidean space $\mathbb{R}^d$. In dimension $d=2$ these connections were introduced and studied by \v{S}evera and Willwacher in \cite{SW}. There, the flat connections mentioned above yield a homotopy between the formality maps for the little disks operad of Kontsevich \cite{K} and Tamarkin \cite{T}, respectively, provided that in the second one the Alekseev--Torossian associator is used. 

In Section 5, we discuss these connections on configuration spaces. We first explain a link between rational homotopy theory and the theory of flat connections with values in $L_\infty$-algebras.
We then describe Kontsevich's model $^*\Graphs_d(n)$ of $\Conf_d(n)$ and the corresponding flat connections $\SW_d(n)$, extending the construction of \v{S}evera and Willwacher to higher dimensions. Finally, we demonstrate how to use this machinery to construct actions of the $\infty$-groupoid of $\Conf_d(n)$ on representations of quadratic differential graded Lie algebras, generalizing the holonomy representations of the braid groups.

\subsection*{Acknowledgements}We thank Alberto Cattaneo, Ya\"el Fr\'egier, Pavol \v{S}evera, and Thomas Willwacher for several helpful conversations related to this project. Moreover, we thank Carlo Rossi for making available his unpublished work with J. Alm \cite{Rossi}. We would also like to thank Utrecht University (C.A.A.) and the University Zurich (F.S.) for their hospitality. Finally, we are grateful to James D. Stasheff, the editor and the referees for their careful revisions and useful comments.

\section{The universal enveloping algebra}

\subsection{Basic definitions}\label{subsection:basic_definitions}
In order to fix notations and conventions, we review the definitions of some functors
and collect relevant facts. We essentially follow \cite{F}.

\begin{definition}
Let $V$ be a graded vector space.
The suspension of $V$, denoted $\s V$, is the graded vector space
$(\s V)^k:=V^{k+1}$.
The desuspension of $V$, denoted $\us V$, is the graded vector space
$(\us V)^k:=V^{k-1}$.
\end{definition}

\begin{definition}
We will make use of the following categories:
\begin{itemize}
 \item The category $\dga_{(a)}$ of (augmented) differential graded algebras
 \item The category $\dgc_{(a)}$ of (co-augmented) differential graded coalgebras
 \item The category $\dgcc_{(a)}$ of (co-augmented) cocommutative differential graded coalgebras
 \item The category $\dgla$ of differential graded Lie algebras
 \item The category $\Linfty$ of $L_\infty$-algebras
\end{itemize}
For the relevant definitions please see \cite{LM,Getzler}.
\end{definition}

\begin{remark}
We will assume that differential graded algebras and differential coalgebras are unital and co-unital, respectively. 
\end{remark}

\begin{definition}
The symmetric coalgebra $\Sym(V)$ of a graded vector space $V$ is the subspace of elements in the tensor coalgebra $\T V$
that are invariant under the action by $\Sigma_{\bullet}$, i.e.,\ the collection of actions of $\Sigma_n$ on $\T^n V$ defined by
$$ \Sigma_n \times \T^n V \to \T^n V, \quad \sigma\bullet (x_1\otimes \cdots \otimes x_n) := (-1)^{|\sigma|} x_{\sigma(1)}\otimes \cdots \otimes x_{\sigma(n)},$$
for $x_1,\dots,x_n \in V$ homogeneous. Here $(-1)^{|\sigma|}$ refers to the
Koszul sign, which is the character of the representation of
$\Sigma_n$ on $\T^n V$ determined by
$$\big( \cdots \otimes x_k \otimes x_{k+1} \otimes \cdots  \mapsto  \cdots \otimes x_{k+1} \otimes x_{k} \otimes \cdots \big) \quad \mapsto \quad (-1)^{|x_k||x_{k+1}|}.$$
There is a natural projection $p: \T V \to \Sym(V)$ given by
$$p(x_1\otimes \cdots \otimes x_n) := \frac{1}{n!} \sum_{\sigma \in \Sigma_n}(-1)^{\sigma}x_{\sigma(1)}\otimes \cdots \otimes x_{\sigma(n)}.$$
The coproduct $\Delta: \T V \to \T V\otimes \T V $, defined via
$$ \Delta(x_1\otimes \cdots \otimes x_n) := \sum_{k=0}^{n} (x_1\otimes \cdots \otimes x_k)\otimes (x_{k+1}\otimes \cdots \otimes x_n),$$
restricts to a graded commutative coproduct on $\Sym(V)$, which we also denote by $\Delta$.
\end{definition}

\begin{definition}
The  Chevalley--Eilenberg functor
$ \CE: \Linfty \to \dgcc_a$
is defined as follows:
\begin{enumerate}
 \item To an $L_{\infty}$-algebra $\g$, the functor $\CE$ associates the co-augmented differential graded cocommutative coalgebra
$(\CE(\g),\delta_g,\Delta)$, where:
\begin{enumerate}
 \item $\CE(\g)$ is the symmetric coalgebra $\Sym(\s\g)$ of the suspension $\s\g$ of $\g$.
The co-unit and co-agumentation are given by the identification $\Sym^0(\s \g)\cong \mathbb{R}$.
 \item The differential $\delta_g$ on $\CE(\g)$ is obtained from the $L_\infty$-structure on $\g$ via the identification $\mathrm{Coder}(\Sym(\s \g))\cong \Hom(\Sym(\s \g),\s \g)$.
\end{enumerate}
\item A morphism of $L_\infty$-algebras $f:\g \to \mathfrak{h}$ is a morphism of differential graded coalgebras
$\CE(f): \CE(\g)\to \CE(\mathfrak{h})$.
\end{enumerate}
\end{definition}

\begin{definition}
The  universal enveloping algebra functor $\U\colon \dgla\hspace{-1pt} \rightarrow\hspace{-1pt} \dga$
is defined as follows:
\begin{enumerate}
 \item To a differential graded Lie algebra $(\g,d,[\cdot,\cdot])$, the
 functor $\U$ associates the differential graded algebra $(\U(\g),d_\U)$, where
\begin{enumerate}
\item $\U(\g)$ is the quotient of the tensor algebra $\T \g$ by the two-sided ideal generated by elements of the form $x\otimes y - (-1)^{|x||y|}y\otimes x - [x,y]$.
\item The differential $d_\U$ on $\U(\g)$ is inherited from  $d_\T: \T \g \to \T \g$, where
 \begin{eqnarray*}
 & d_\T(x_1\otimes \cdots \otimes x_n) & := \\
 & \sum_{i=1}^n(-1)^{|x_1|+\cdots + |x_{i-1}|}&x_1\otimes \cdots \otimes x_{i-1}\otimes dx_i\otimes x_{i+1}\otimes \cdots \otimes x_n.
 \end{eqnarray*}
\end{enumerate}
\item To a morphism $f:\g\to \mathfrak{h}$ of differential graded Lie algebras, the functor $\U$ associates the morphism
$ \U(f): \U(\g) \to \U(\mathfrak{h})$
induced by
$$ \T(f): \T \g \to \T \mathfrak{h}, \quad \T(f)(x_1\otimes \cdots\otimes x_n) := f(x_1)\otimes \cdots \otimes f(x_n).$$
\end{enumerate}

The (anti)symmetrization functor
$\Sigma: \dga \to \dgla$
maps $(A,d,\cdot)$ to the differential graded Lie algebra $\Sigma A$, whose underlying complex is $(A,d)$
and whose Lie bracket is defined by setting $ [x,y] := x\cdot y - (-1)^{|x||y|} y \cdot x$.
The functor $\Sigma: \dga \to \dgla$ is right adjoint to $\U: \dgla \to \dga$ and $\U$ preserves quasi-isomorphisms.
\end{definition}

\begin{definition}
Let $(C,d,\Delta)$ be a co-augmented differential graded coalgebra.
The {\em reduced coproduct} $\overline{\Delta}$ is defined on the kernel $\overline{C}$ of the co-unit map via
$$\overline{\Delta}(x) := \Delta(x) - x\otimes 1 - 1\otimes x.$$
\end{definition}

\begin{definition}\label{def:cobar}
The  cobar functor
$\Omega: \dgc_a \to \dga_a$
is defined as follows:
\begin{enumerate}
 \item To a co-augmented differential graded coalgebra $(C,d,\Delta)$, the functor $\Omega$ associates the augmented differential graded
algebra $(\Omega(C),\delta,\cdot)$, where:
\begin{enumerate}
\item The underlying augmented graded algebra is the tensor algebra $\T(\us \overline{C})$ of the desuspension $\us \overline{C}$.
\item The differential $\delta$ of $\Omega(C)$ is  determined by
$\delta(\us x) := \us dx + \partial(\us x),$ where $\partial(\us x) = -\sum_i (-1)^{|x_i|} \us x_i\otimes \us y_i$ if $\overline{\Delta}(x) = \sum_i x_i\otimes y_i$.
\end{enumerate}
\item To a morphism $f: C\to D$ of augmented differential graded cocommutative coalgebras, the functor $\Omega$ associates
the morphism
$\Omega(f): \Omega(C)\to \Omega(D)$
induced by $\T(\us f)$.
\end{enumerate}
\end{definition}

\begin{definition}
The  bar functor $\B: \dga_a \to \dgc_a$
is defined as follows:
\begin{enumerate}
\item
To an augmented differential graded algebra $(A,d,\cdot)$, the functor $\B$ associates the co-augmented differential graded coalgebra $(\B(A),\delta,\Delta)$, where:
\begin{enumerate}
\item The underlying augmented graded coalgebra is the tensor coalgebra $\T(\s \underline{A})$ of the suspension $\s \underline{A}$ of the augmentation ideal $\underline{A}$.
\item The differential $\delta$ of $\B(A)$ is the coderivation given by
\begin{eqnarray*}
\delta(\s x_1 \otimes \dots \otimes \s x_k) &:=& -\sum_{i=1}^k (-1)^{n_i}(\s x_1 \otimes \dots \otimes \s dx_i \otimes \dots  \otimes \s x_k)\\
&&+ \sum_{i=2}^k (-1)^{n_i} \s x_1 \otimes \dots \otimes \s (a_{i-1} a_i) \otimes\dots \otimes  \s x_k,
\end{eqnarray*}
where $n_i:=|\s x_1|+\dots +|\s x_{i-1}|$
on homogeneous elements of $\underline{A}$.
\end{enumerate}
\item To a morphism $f: A\to A'$ of augmented differential graded algebras, the functor $\B$ associates
the morphism
$\B f: \B A \to \B A'$
induced by $\T(\s f)$.
\end{enumerate}
\end{definition}

\begin{remark}
In applications the bar complex is not sufficient and it has to be replaced by
the completed bar complex; see Appendix~\ref{section:twisting_cochains} for details.
\end{remark}

\begin{definition}
The  Lie functor
$\LF: \dgcc_a \to \dgla$
is defined as follows:
\begin{enumerate}
 \item To a co-augmented differential graded cocommutative coalgebra $(C,d,\Delta)$, the functor $\LF$ associates the  differential graded
algebra $(\LF(C),\delta,[,])$, where:
\begin{enumerate}
\item The underlying graded Lie algebra is the free graded Lie algebra on the  desuspension $\us \overline{C}$ of $\overline{C}$.
\item The differential $\delta$ on $\LF(C)$ is the Lie derivation determined by
\begin{align*}
\delta(\us x) := \us dx + \partial(\us x) \in \LF(C) \subset \T(\overline{C}),
\end{align*}
on homogeneous elements of $\overline{C}$, where $\partial(\us x) = -\sum_i (-1)^{|x_i|} \us x_i\otimes \us y_i$ if $\overline{\Delta}(x) = \sum_i x_i\otimes y_i$. Note that the cocommutativity of the coproduct guarantees that the right hand side belongs to $\LF(C)$.
\end{enumerate}
\end{enumerate}
\end{definition}

The following theorem  will be essential for our construction.

\begin{theorem}[Quillen \cite{Quillen}, Hinich \cite{Hinich}]\label{Quillen}
The functor $\LF\colon \dgcc_a \rightarrow \dgla$ is left adjoint to
$\CE\colon \dgla \rightarrow \dgcc_a$.
Moreover, the adjunction maps $X\to \CE(\LF(X))$ and $\LF(\CE(\g)) \to \g$ are quasi-isomorphisms.
\end{theorem}

\begin{remark}
The above theorem works under the hidden assumption that we restrict to the subcategory of \emph{connected} differential
graded cocommutative coalgebras; see Appendix B of \cite{Quillen}. All the coalgebras to which we will apply the Theorem
are of this kind.\footnote{In contrast, the coalgebra $C_\bullet(M)$ is not connected. This is what forces us to introduce
the completed bar complex; see Appendix~\ref{section:twisting_cochains}.}
\end{remark}

\begin{definition}
The strictification functor
$\ST: \Linfty \to \dgla$
is $\ST:= \LF\circ \CE$.
\end{definition}

\begin{corollary}\label{corollary:g_to_ST(g)}
Let $\mathfrak{g}$ be an $L_\infty$-algebra. Then the unit of the adjunction between $\LF$ and $\CE$, applied to $\CE(\g)$ gives a
map
\[ \eta \in \Hom_{\dgcc_a}(\CE(\g), \CE(\ST(\g)))\cong \Hom_{\Linfty}(\g,\ST(\g)),  \]
which is a quasi-isomorphism of $L_\infty$-algebras.
\end{corollary}

\begin{remark}\label{remark:ST(g)_to_g}
In case $\g$ is a differential graded Lie algebra, there is also a morphism
$ \rho: \ST(\g) \to \g$
of differential graded Lie algebras obtained by the adjunction
\[ \Hom_{\dgla}(\ST(\g),\g) = \Hom_{\dgla}(\LF(\CE(\g)),\g) \cong \Hom_{\dgcc_a}(\CE(\g),\CE(\g))\]
from $\id_{\CE(\g)}$. Moreover, $\rho \circ \eta = \mathsf{id}$ holds, hence $\rho$ is a quasi-isomorphism.
\end{remark}

\begin{remark}
Let $\iota$ denote the inclusion functor $\dgcc_a \rightarrow \dgc_a$.
One can check that the functors $\Omega \circ \iota$ and $\U \circ \LF$
are isomorphic.

We sum up this subsection in the diagram
\begin{align*}
\xymatrix{
\Linfty \ar[rr]^{\CE} \ar[rrrd]_{\ST}& & \dgcc_a \ar[r]^{\iota} \ar[dr]^{\Lie} & \dgc_a  \ar[r]^{\Omega}& \dga_a\\
 &  & & \dgla. \ar[ur]_{\U}&
}
\end{align*}
Observe that the triangle on the left commutes, while the triangle on the right side commutes up
to a natural isomorphism.
\end{remark}

\subsection{The enveloping algebra}\label{subsection:universal_enveloping}

Following \cite{B}, we now define the universal enveloping algebra of an $L_\infty$-algebra.
The idea is to use the strictification functor.
\begin{definition}
The universal enveloping functor
$\Uinfty\colon \Linfty \rightarrow \dga$
is given by\linebreak
$\Uinfty:=\U \circ \ST$.
We call $\Uinfty(\mathfrak{g})$ the universal enveloping algebra of $\mathfrak{g}$.
\end{definition}

The universal enveloping algebra $\Uinfty(\g)$ of a differential graded Lie algebra $\g$, seen as an $L_\infty$ algebra, is not
the same as the usual enveloping algebra $\U(\g)$ of $\g$. However, these two algebras are naturally quasi-isomorphic:

\begin{proposition}\label{Halperin}
Let $\mathfrak{g}$ be a differential graded Lie algebra.
The map
$\U (\rho)\colon$\linebreak $\Uinfty(\g) \rightarrow \U(\g)$
induced from $\rho: \ST(\g) \to \g$
is a quasi-isomorphism of differential graded algebras.
\end{proposition}

\begin{proof}
This is immediate from the fact that $\rho$ is a quasi-isomorphism and that the functor
$\U$ preserves quasi-isomorphisms.
\end{proof}

As in the usual case of differential graded Lie algebras, the functor $\Uinfty$ can be characterized as a left adjoint
to a forgetful functor:

\begin{proposition}\label{prop:U_left_to_Sigma}
The functor $\Uinfty\colon \Linfty \rightarrow \dga$ is left adjoint to the
forgetful functor $\Sigma_\infty:=\iota \circ \Sigma\colon \dga \rightarrow \Linfty,$
where $\iota\colon \dgla \rightarrow \Linfty$ is the inclusion functor.
\end{proposition}
\begin{proof}
This is a formal consequence of the adjunctions discussed above:
\begin{eqnarray*}
\Hom_{\dga}(\Uinfty(\g), A)&\cong& \Hom_{\dga}(\U(\ST(\g)), A)
 \cong \Hom_{\dgla}(\ST(\g), \Sigma(A)) \\
 &\cong& \Hom_{\dgcc_a}( \CE(\g),\CE(\Sigma(A)))
 = \Hom_{\Linfty}(\g, \Sigma_\infty(A)).
\end{eqnarray*}
\end{proof}

The proof of the following lemma will be omitted for brevity.

\begin{lemma}
The functor $\Uinfty\colon \Linfty \rightarrow \dga$ preserves quasi-isomorphisms.
\end{lemma}



\section{Complete $L_\infty$-algebras}

\subsection{Generalities about complete $L_\infty$-algebras}\label{subsection:filtered}

The computation of holonomies is an operation that involves infinite sums. For this reason, we have to consider $L_\infty$-algebras where infinite sums can be treated.

\begin{definition}
An ideal  of an $L_\infty$-algebra $\g$ is a graded subspace $I\subset \g$ such that
\[[x_1,\dots, x_k] \in I ,\textrm{\, if one of the $x_i$ belongs to $I$}.\]
A filtration $F$ on $\g$ is a decreasing sequence of ideals
$F_1(\g)=\g \supseteq F_2(\g) \supseteq F_3(\g)\supseteq \cdots, $ such that:
\begin{enumerate}
\item $\bigcap_k F_k(\g)=0$.
\item If $x_{i} \in F_{l_i}(\g)$, then
$[x_1,\dots,x_k]\in F_{l_1+\dots +l_k}(\g)$.
\end{enumerate}

\end{definition}

\begin{definition}
A filtered $L_\infty$-algebra is an $L_\infty$-algebra together with a filtration.

If $\g, \mathfrak{h}$ are filtered $L_\infty$-algebras, a filtered morphism is a morphism $\phi$ such that if $x_i \in F_{l_i}(\g)$ then $\phi_k(x_1,\dots ,x_k) \in F_{l_1+\dots +l_k}(\mathfrak{h})$.
\end{definition}

\begin{remark}
\hspace{0cm}
\begin{enumerate}
\item If $I$ is an ideal of $\g$, then the quotient space $\g/I$ inherits the structure of an $L_\infty$-algebra.

\item Given a filtered $L_\infty$-algebra $\g$, there is a diagram

\[0 \leftarrow \g/F_2(\g) \leftarrow  \g/F_3(\g) \leftarrow \cdots . \]
The completion of $\g$, denoted $\hat{\g}$, is the limit

\[\hat{\g} := \varprojlim \g/ F_k(\g).\]

The natural map $\iota\colon  \g \rightarrow \hat{\g}$ given by $x \mapsto (\overline{x},\overline{x},\overline{x},\dots)$ is an injection in
view of the first property of the definition of a filtration.
\end{enumerate}
\end{remark}

For the sake of brevity, we will omit the proof of the following lemma.
\begin{lemma}\label{functoriality of completion}
The completion $\g \mapsto \hat{\g}$ defines a functor on the category of
filtered $L_\infty$-algebras and filtered morphisms. Moreover, for a
filtered morphism $\phi\colon \g \rightarrow \mathfrak{h}$, the following holds: $\iota \circ \phi=\hat{\phi}\circ \iota$.

\end{lemma}







\begin{definition}\label{definition:completeness}
A filtered $L_\infty$-algebra $\g$ is complete if the canonical
injection $\g \to \hat{\g}$ is  an isomorphism.
\end{definition}

\begin{remark}
\hspace{0cm}
\begin{enumerate}
\item A  filtered $L_\infty$-algebra $\g$ has the structure of a topological vector space where the sequence $F_k(\g)$ is a local basis for $0\in \g$. This topology is Hausdorff, since it is induced by the metric:
$d(x,y):= \mathsf{\inf}\{\frac{1}{k}: x-y \in F_k(\g)\}$.

In particular, any sequence of elements in a filtered $L_\infty$-algebra $\g$ has at most one limit.
In case $\g$ is complete in the sense of Definition~\ref{definition:completeness}, it is also complete
as a topological vector space.

\item Following \cite{Getzler}, we observe that there is a natural decreasing sequence of ideals on any $L_\infty$-algebra $\g$, defined recursively as follows: $F_1(\g):=\g$ and

\[F_k(\g):=  \sum_{l_1+ \dots +l_i=k} [F_{l_1}(\g),\dots ,F_{l_i}(\g)]. \]

In \cite{Getzler} this filtration is called the lower central filtration of $\g$.
Since it might fail to be a filtration in our sense, because the intersection of the $F_k(\g)$ might not be zero,
we refer to the collection $F_k(\g)$ as the lower central series of $\g$.

Given any  filtration $F'$ on $\g$, it is clear that
$ F_k(\g)\subseteq F'_k(\g),$
and therefore
\[\bigcap_k F_k(\g)\subseteq \bigcap_kF'_k(\g) =0.\]
Thus, if $\g$ admits a filtration at all, then the lower central series is a filtration, and
it is the minimal one.
\end{enumerate}
\end{remark}

\begin{definition}\label{definition:MC}
A Maurer--Cartan element of a complete $L_\infty$-algebra is an element $\alpha \in \mathfrak{g}^1$ such that
${\displaystyle
\sum_{k\geq 1}} \frac{1}{k!}[\underbrace{\alpha \otimes \dots \otimes \alpha}_{k \mathsf{ \,times}}]=0. $
We denote by $\mathsf{MC}(\g)$ the set of all Maurer--Cartan elements of $\g$.
\end{definition}

\begin{lemma}
Let $\phi\colon \g \rightarrow \mathfrak{h}$ be a filtered morphism between complete $L_\infty$-algebras.
There is a map of sets
$\phi_*\colon \g \rightarrow \mathfrak{h}$,
given by the formula $\phi_*(\alpha):= \sum_{k\geq 1} \phi_k(\alpha^{\otimes k})$.
This map is continuous at zero and preserves Maurer--Cartan elements.

\end{lemma}
\begin{proof}
Since $\phi$ is a filtered morphism, we know that $\phi_k(\alpha^{\otimes k})\in F_k(\mathfrak{h})$ and therefore the sum converges.
It is clear that if $\alpha \in F_k(\g)$ then $\phi_*(\alpha)\in F_k(\mathfrak{h})$, so that the map is continuous at zero.
Let us now prove that $\phi_*(\alpha)$ is a Maurer--Cartan element whenever
$\alpha$ is as follows:

\begin{eqnarray*}
\sum_{k\geq 1} \frac{1}{k!}[\phi_*(\alpha) \otimes \dots \otimes \phi_*(\alpha)]&=&\sum_{k\geq 1}\frac{1}{k!}\sum_{l_1,\dots l_k} [\phi_{l_1}(\alpha^{\otimes l_1}) \otimes \dots \otimes \phi_{l_k}(\alpha^{\otimes l_{k}})]\\
&=&\sum_{p\geq 1}\sum_{l_1+\dots +l_k=p} \frac{1}{k!} [\phi_{l_1}(\alpha^{\otimes l_1}) \otimes \dots \otimes \phi_{l_k}(\alpha^{\otimes l_{k}})]\\
&=&\phi_*\left(\sum_{k\geq 1}\frac{1}{k!}[\alpha \otimes \dots \otimes \alpha]\right)=\phi_*(0)=0.
\end{eqnarray*}
\end{proof}

\begin{remark}\label{remark:MC_dga}
Similarly, for $A$ a differential graded algebra, one defines the set of
Maurer--Cartan elements to be
$ \mathsf{MC}(A) := \{ \alpha \in A^1: d\alpha + \alpha \cdot \alpha = 0\}$.
\end{remark}

\subsection{Compatibility with various functors}

The proof of the following lemma will be omitted for brevity.
\begin{lemma}\label{operations on filtrations}
Suppose that $V$ is a filtered graded vector space. Then we have the
following:

\begin{itemize}

\item The reduced tensor algebra $\overline{\T}V$ is a filtered algebra with filtration:

\[F_k(\overline{\T} V)= \sum_{l_1+\dots +l_r\geq k} F_{l_1}(V)\otimes \cdots \otimes F_{l_r}(V).\]

\item The vector space $\overline{\Sym}(V)$ is also a filtered graded vector space with filtration:
\[F_k(\overline{\Sym}(V)):= \langle \{x_1 \otimes \dots \otimes x_r \in \overline{\Sym}(V):\ \exists\  l_1+\dots +l_r\geq k \text{ with } x_i \in F_{l_i}(V)\}   \rangle\]
\item The free graded Lie algebra $\LF(V)$ is a filtered Lie algebra with filtration:
\[F_k(\LF(V)):= \langle \{ P(x_1, \dots x_r) \in \LF(V):\ \exists\  l_1+\dots +l_r\geq k \text{ with } x_i \in F_{l_i}(V)\}   \rangle.\]
Here $P(x_1, \dots,x_r)$ denotes a Lie monomial of length $k$ on $x_1, \dots ,x_r$ where all the $x_i$ appear.
\end{itemize}
\end{lemma}









We now prove that the strictification of $L_\infty$-algebras is compatible with filtrations.

\begin{lemma}\label{lemma:filtration_strictification}
Let $\g$ be a filtered $L_\infty$-algebra. Then the differential graded algebra $\ST(\g)$ has an induced filtration and
the natural morphism $\eta\colon \g \rightarrow \ST(\g)$
is a filtered morphism.
\end{lemma}

\begin{proof}
Recall that the Lie algebra $\ST(\g)$ is the free Lie algebra on the vector space $V=\us \overline{\Sym}(\s \g)$.  In view of
Lemma~\ref{operations on filtrations},  we know that there is a filtration on $\ST(\g)$ seen as a Lie algebra.
We need to prove that this filtration is compatible with the differential, i.e.,\ that
 $\delta (F_k(\ST(\g)))\subset F_k(\ST(\g)). $
Since $\delta$ is a derivation with respect to the Lie bracket, it suffices to prove the claim for elements of $V$.
The differential $\delta$ is the sum of two coboundary operators: one induced from that of $\g$ and one induced from the
coproduct. The claim is clearly true for the first differential. Let us prove it for the differential that comes from the coproduct, given by
\[\us (\s {x}_1\otimes \dots \otimes \s {x}_{n})\mapsto  -\hspace{-2pt} \sum_i (-1)^{|x_1|+ \dots +|x_i|+i}\us (\s {x}_1\otimes \dots \otimes \s x_i)  \otimes \us(\s x_{i+1}\otimes \dots \otimes \s {x}_{n}).\]
Since the right-hand side is the sum of Lie monomials on the same
elements, we conclude that if the left-hand side belongs to $F_k(\ST(\g))$,
so does the right-hand side.

So far, we have seen that the differential graded Lie algebra $\ST(\g)$
inherits a filtration; it remains to show that the map $\eta\colon \g \rightarrow \ST(\g)$ is a filtered map. The components of this map are given by the formula
\begin{eqnarray*}
 \eta_k(&x_1&\otimes \dots \otimes x_k)=\pm \s\us(\s x_1 \otimes \dots \otimes \s x_k) \\
&&+  \sum_{k_1+k_2=k}\sum_{\sigma \in (k_1,k_2)} \pm \s\us(\s x_{\sigma(1)}\otimes \cdots \otimes \s x_{\sigma(k_1)})\otimes \s \us (\s x_{\sigma(k_1+1)}\otimes \cdots \otimes \s x_{\sigma(k)}) \\
&& + \cdots,
\end{eqnarray*}
and therefore if $x_i \in F_{l_i}(\g)$ then $\eta_k(x_1\otimes \dots \otimes x_k)\in F_{l_1 + \dots + l_k}(\ST(\g)),$
and we conclude that $\eta$ is a filtered map.
\end{proof}

\begin{definition}
A filtration of an augmented differential graded algebra $A$ is a filtration of its augmentation ideal.
A filtered augmented differential graded algebra $A$ is an augmented differential graded algebra with a filtration.
\end{definition}


\begin{lemma}\label{lemma:filtration_enveloping}
The universal enveloping functor $\U: \dgla \to \dga$ extends to a functor from the category of filtered differential graded Lie algebras to the category of filtered differential graded algebras as follows:
\begin{enumerate}
 \item For $\g$ a filtered differential graded Lie algebra, the augmentation ideal of $\U(\g)$ carries the filtration inherited from $\T\g$.
 \item For $f: \g\to \mathfrak{h}$ a filtered morphism of differential graded Lie algebras algebras, the morphism $\U(f): \U(\g)\to \U(\mathfrak{h})$
is a filtered morphism.
\end{enumerate}
\end{lemma}

\begin{proof}
This follows from the definitions and the fact that the expression $x\otimes y - (-1)^{|x||y|}y\otimes x -[x,y]$ lies in $F_{k+l}(\overline{\T}\g)$
for $x\in F_k(\g)$ and $y\in F_l(\g)$.
\end{proof}

\begin{corollary}
The universal enveloping algebra $\Uinfty(\g)$ of a filtered $L_\infty$-algebra $\g$ is naturally a filtered augemented differential graded algebra.
\end{corollary}

\begin{proof}
This is a direct consequence of  Lemmas~\ref{lemma:filtration_strictification} and~\ref{lemma:filtration_enveloping}.
\end{proof}



\begin{remark}
\hspace{0cm}
\begin{enumerate}
\item Recall from \cite{Getzler} that if $\g$ is an $L_\infty$-algebra and $A$ is a differential
graded commutative algebra then the tensor product $\g \otimes A$ is an $L_\infty$-algebra with brackets:
\begin{equation*}
\begin{cases}
[x \otimes a]=[x]\otimes a +(-1)^{|x|+1}x\otimes da,\\
[x_1 \otimes a_1,\dots , x_k \otimes a_k]=(-1)^{\sum_{i<j}|a_i|(|x_j|+1)}[x_1,\dots,x_k]\otimes a_1\dots a_k, \quad k\neq 1.

\end{cases}\hspace*{-2.3pt}
\end{equation*}
Observe that $-\otimes A$ extends to a functor: given a morphism $\gamma$ of $L_\infty$-algebras, one defines
$ \gamma\otimes \id_A: \g\otimes A \to \mathfrak{h}\otimes A$
to be given by the structure maps
 \begin{eqnarray*}
 &&(\gamma\otimes \id_A)_k((\s x_1 \otimes a_1)\otimes \cdots \otimes (\s x_k\otimes a_k)) :=\\ && \quad (-1)^{\sum_{i<j}|a_i|(|x_j|+1)} \gamma_k(\s x_1\otimes \cdots \s x_k) \otimes (a_1 \cdots a_k),
 \end{eqnarray*}
where we see an element $\s x\otimes a$ in $\s (\g \otimes A)$ via the map
$$ \s(\g \otimes A) \cong \s \g \otimes A, \quad \s (x\otimes a) \mapsto \s x \otimes a.$$

\item If $\g$ is filtered, $\g \otimes A$ is a filtered $L_{\infty}$-algebra with filtration:
$F_k(\g \otimes A):=F_k(\g)\otimes A$.
Moreover, if $\gamma: \g \to \mathfrak{h}$ is a morphism of filtered $L_\infty$-algebras, so is $\gamma\otimes \id_A$.

The operation $-\otimes A$ is functorial and so---see
Lemma~\ref{functoriality of completion}---we have a commutative diagram:

\[
\xymatrix{
\g \otimes A \ar[r]^{\iota \otimes A} \ar[d]_\iota& \hat{\g}\otimes A\ar[d]^\iota\\
\widehat{(\g \otimes A)}\ar[r]^{\widehat{\iota \otimes A}}& \widehat{ (\hat{\g}\otimes A)}.
}
\]
\item Similar statements apply if one replaces $\g$ by a (filtered) differential graded algebra and drops the commutativity of $A$.
\end{enumerate}
\end{remark}

The following lemma is straightforward to check:

\begin{lemma}\label{lemma:denseness}
\hspace{0cm}
\begin{itemize}
\item Let $V$ be a filtered graded vector space. Then $\T V$ is dense in $\T \hat{V}$.
\item Let $\g$ be a filtered differential graded Lie algebra. Then $\U(\g)$ is dense in $\U(\hat{\g})$.
\item Let $\g$ be a filtered $L_\infty$-algebra. Then $\ST(\g)$ is dense in $\ST(\hat{\g})$.
Moreover, if $A$ is a commutative differential graded algebra, then $\g\otimes A$ is dense in $\hat{\g}\otimes A$.
\end{itemize}
\end{lemma}

\begin{corollary}\label{corollary:denseness}
\hspace{0cm}
\begin{itemize}
\item Let $V$ be a filtered graded vector space. Then $\widehat{\T V}$ is naturally isomorphic to $\widehat{\T \hat{V}}$.
\item Let $\g$ be a filtered differential graded Lie algebra. Then $\widehat{\U(\g)}$ is naturally isomorphic to $\widehat{\U(\hat{\g})}$.
\item Let $\g$ be a filtered $L_\infty$-algebra. Then $\widehat{\ST(\g)}$ is naturally isomorphic to $\widehat{\ST(\hat{\g})}$.
Moreover, if $A$ is a commutative differential graded algebra, then $\widehat{\g\otimes A}$ is naturally isomorphic to $\widehat{\hat{\g}\otimes A}$.

\end{itemize}
\end{corollary}

\begin{proof}
This follows from Lemma~\ref{lemma:denseness} and the fact that all the maps
$\T V \to \T \hat{V}$, $\U(\g) \to \U(\hat{\g})$, $\ST(\g)\to
\ST(\hat{\g})$, and $\g\otimes A \to \hat{\g}\otimes A$
are inclusions. This is obvious except for $\U(\g) \to \U(\hat{\g})$.
We are done if we can prove that for any graded Lie subalgebra $i: \mathfrak{h}\to \g$,
the induced map $\U(i): \U(\mathfrak{h}) \to \U(\g)$ is injective. But this is the case if
$$\mathrm{gr}\U(i): \mathrm{gr}\U(\mathfrak{h})  \to \mathrm{gr}\U(\mathfrak{g})$$
is injective. Here $\mathrm{gr}$ denotes the functor that maps a filtered vector space to its associated graded
and $\U(\g)$ is seens as a filtered vector space with the filtration whose members $\mathfrak{F}_k\U(\g)$ are
the images of $\T^{\le k}(\g)$ under the quotient map.\footnote{Strictly speaking, this kind of filtration is opposite to the way we defined them.}
However, $\mathrm{gr}\U(\g)$ is canonically isomorphic to $\mathsf{S}(\g)$, the graded symmetric
algebra of $\g$, and $\mathrm{gr}\U(i)$ corresponds to $\mathsf{S}(i)$.
It is clear that $\mathsf{S}(i)$ is injective.
\end{proof}

\begin{definition}
Given\hspace{-0.5pt} a filtered\hspace{-0.5pt} $L_\infty$-algebra\hspace{-0.5pt} $\g$\hspace{-0.5pt} and\hspace{-0.5pt} a commutative differential graded algebra $A$,
we denote the completion of the $L_\infty$-algebra $\g\otimes A$ by $\g \hat{\otimes} A$.

Given a filtered $L_\infty$-algebra $\g$, we denote the completion of the universal enveloping algebra $\Uinfty(\g)$ by $\hatUinfty(\g)$.
\end{definition}

\begin{definition}
Let $V$ be a graded vector space and $W$ be a filtered vector space.

The graded vector space $\Hom(V,W)$ carries a filtration defined by
$$ \phi \in F_k\Hom(V,W) \quad :\Leftrightarrow \quad \mathrm{im}(\phi) \subset F_kW.$$
\end{definition}

\begin{lemma}
Let $V$ be a graded vector space and $W$ a complete vector space.
Then $\Hom(V,W)$ is complete.
\end{lemma}

\begin{proof}
Given a Cauchy sequence $\phi_i$ in $\Hom(V,W)$, we define $\phi: V\to W$ via
$$\phi(x):= \lim_{i \to \infty} \phi_i(x).$$
By definition of the filtration on $\Hom(V,W)$, the sequence $\phi_i(x)$ will be Cauchy and since $W$ is complete,
$\phi(x)$ is well-defined.
Because $W$ is a topological vector space with respect to the topology induced from the filtration,
the map $\phi$ is a linear map.
\end{proof}

\begin{remark}
Given a graded vector space $W$ and a filtered graded vector space $V$, there is a natural inclusion of filtered graded
vector spaces
$W^*\otimes V \to \Hom(V,W)$.
The above lemma implies that the completion $\widehat{W^*\otimes V}$ can be naturally
identified with a subspace of $\Hom(V,\hat{W})$.
\end{remark}

\section{Parallel transport}

\subsection{$\Ainfty$ de Rham Theorem} \label{subsection:deRham}

We briefly describe an $\Ainfty$ version of de Rham's theorem that is due to Gugenheim \cite{G}. It is the key
ingredient in the definition of higher holonomies in the next subsection.
The construction involves a family of maps from cubes to simplices  introduced by Chen \cite{Chen}. We use the maps given by Igusa in \cite{I}.
Let us now recall Gugenheim's morphism from \cite{G}, following the conventions of \cite{AS}, where the
interested reader can find more details.

Let $M$ be a smooth manifold, and denote by $\Path M$ the path space of $M$.
The first ingredient for the $\Ainfty$ de Rham theorem is Chen's map
\begin{align*}
\Chen: \overline{\B \Omega(M)} =\bigoplus_{k\ge 1} \left(\s \Omega(M) \right)^{\otimes k} \to \Omega(\Path M).
\end{align*}
It is a linear map of degree $0$ and constructed as follows: We denote the evaluation map
\begin{eqnarray*}
\Path M \times  \Delta_k  &\rightarrow& M^k,\\
(\gamma,(t_1,\dots,t_k))&\mapsto &(\gamma(t_1),\dots, \gamma(t_k))
\end{eqnarray*}
by $\mathrm{ev}$ and the natural projections $\Path M \times {\Delta}_k \rightarrow \Delta_k$ and  $M^{ k}\rightarrow M$ by $\pi$ and $p_i$, respectively.
Chen's map is
\begin{align*}
\Chen(\s a_1\otimes \cdots \otimes \s a_k):= (-1)^{\sum_{i=1}^n[a_i](k-i)}\pi_{*}(\mathrm{ev})^{*}(p_1^*a_1\wedge \cdots \wedge p_n^*a_k),
\end{align*}
where $[a_i]$ is the degree of $\s a_i\in \s \Omega(M)$.

The next step in the construction of the $\Ainfty$ de Rham theorem is a special sequence of maps from the cubes to the simplices.
We follow a construction due to Igusa \cite{I} and make use of the following definition of the $k$-simplex
\begin{align*}
\Delta_k:=\{(t_1,\dots,t_k)\in \mathbb{R}^{k}: 1 \ge t_1 \ge t_2 \ge \cdots \ge t_k \ge 0\} \subset \mathbb{R}^{k}.
\end{align*}

\begin{definition}[Igusa]
For each $k\geq 1$, the map
\[\Theta_{(k)}\colon  I^{k-1} \rightarrow \Path \Delta_k, \]
 is defined to be the composition
 \begin{align*}
\xymatrix{
I^{k-1} \ar[r]^{\lambda_{(k)}}& \Path I^k\ar[r]^{\Path \pi_k}& \Path \Delta_k.
}
\end{align*}

Here $\pi_k\colon  I^k \rightarrow \Delta_k$ is given by $\pi_k(x_1,\dots,x_k):=(t_1,\dots,t_k)$, with components
\[t_i:= \mathrm{ max }\{x_i,\dots, x_k\}.\]
The map $\lambda_{(k)}\colon I^{k-1}\rightarrow \Path I^k$
is defined by sending a point $(x_1,\dots,x_{k-1})$ to the path which goes backwards through the following $k+1$ points:
\[0 \leftarrow  x_1 e_1\leftarrow  \dots \leftarrow  ( x_1e_1+\dots + x_{k-1}e_{k-1})\leftarrow  (x_1e_1+\dots + x_{k-1}e_{k-1}+e_k), \]
where $(e_1,\cdots, e_k)$ denotes the standard basis of $\mathbb{R}^{k}$.
In other words, for $j=0,\dots, k$ we set
\[\lambda_{(k)}(x_1,\dots, x_{k-1})(\frac{k-j}{k})= x_1 e_1+ \dots +x_{j} e_{j},   \]
where $x_k=1$, and interpolate linearly.

By convention, $\Theta_{(0)}$ is the map from a point to a point.

We denote the map adjoint to $\Theta_{(k)}$ by $\Theta_k\colon  I^k \rightarrow \Delta_k$.
\end{definition}

\begin{definition}
The map $\mathsf{S}: \Omega(\Path M) \to \s C^{\bullet}(M)$ is the composition of
\begin{eqnarray*}
\Omega(\Path M) &\to& C^{\bullet}(M),\\
\alpha &\mapsto& \left( \sigma \mapsto \int_{I ^{k-1}}(\Theta_{(k)})^{*}\Path \sigma^{*}\alpha\right).
\end{eqnarray*}
\end{definition}

\begin{definition}
Given a smooth manifold $M$ and an integer $n\geq 1$, we define the map
$ \psi_n: \left(\s \Omega(M)\right)^{\otimes n} \to \s C^{\bullet}(M)$,
as follows:
\begin{enumerate}
\item For $n=1$, we set:
$\left(\psi_1(\s a)\right) (\sigma:\Delta_k \to M):= (-1)^{k} \left(\int_{\Delta^{k}}\sigma^{*}a \right)$.
\item For $n>1$, we set
$\psi_n(\s a_1 \otimes \cdots \otimes \s a_n) := (\mathsf{S}\circ \Chen)(\s a_1\otimes \cdots \otimes \s a_n)$.

\end{enumerate}
\end{definition}

\begin{remark}
Observe that $\psi_1(\s a)$ coincides with $(\mathsf{S}\circ \Chen)(\s a)$, except for the case
when $a$ is of degree $0$, i.e.,\ a function. In that case, $(\mathsf{S} \circ \Chen)(\s a) = 0$, while
\begin{align*}
\left(\psi_1(\s a)\right)(\sigma: \{*\} \to M) := a(\sigma(0)).
\end{align*}
\end{remark}

\begin{theorem}[Gugenheim]\label{theorem:A_infty_quasi-isomorphism}
The sequence of maps
$\psi_n\colon \left(\s \Omega(M)\right)^{\otimes n} \to \s
C^\bullet(M)$\linebreak
defines an $A_\infty$-morphism from $(\Omega(M),-d,\wedge)$ to differential graded algebra of smooth singular cochains $(C^{\bullet}(M),\delta,\cup)$.
Moreover, this morphism is a quasi-isomorphism and the construction is natural with respect to pull backs along smooth maps.
\end{theorem}

\subsection{Holonomies}

Using the constructions given above, it is now a simple task to define holonomies  for connections with values in $L_\infty$-algebras.


\begin{lemma}\label{lemma:coefficients}
Let $\g$ be an $L_\infty$-algebra and $A$ a commutative differential graded algebra.
Then there is
a natural map of differential graded algebras
\[\tau\colon  \Uinfty(\g \otimes A)\rightarrow \Uinfty(\g)\otimes A.\]
This map is given on generators of the free algebra   $\Uinfty(\g \otimes A)$ by the formula
\begin{eqnarray*}
&&\us \Big(\s (x_1 \otimes a_1)\otimes \dots \otimes \s(x_k \otimes a_k)\Big)\mapsto \\
&& (-1)^{ \sum_{i<j}|a_i|(|x_j|+1)} \us \Big((\s x_1 \otimes \dots \otimes  \s x_k) \otimes (a_1 \dots a_k) \Big).
 \end{eqnarray*}

Moreover, if $\g$ is filtered then $\tau$ is a filtered map.
\end{lemma}

\begin{proof}
First recall that there is a natural morphism of $L_{\infty}$-algebras
$\eta$ from $\g$ to its strictification $\ST(\g)$. The adjunction property of $\Uinfty$ yields a morphism
$$\gamma \in \Hom_{\dgla}(\ST(\g),\Sigma(\Uinfty(\g))) \cong \Hom_{\dga}(\Uinfty(\g),\Uinfty(\g)) $$
corresponding to the identity of $\Uinfty(\g)$.

The composition of $\eta$ and $\gamma$ is an $L_\infty$-morphism from $\g$ to $\Sigma(\Uinfty(\g))$. Tensoring with $\id_A$
yields an $L_\infty$-morphism
$$ (\gamma\circ \eta)\otimes \id_A: \g\otimes A \to \Sigma(\Uinfty(\g))\otimes A.$$
Using the adjunction properties, as well as the natural isomorphism $\Sigma(C\otimes A) \cong \Sigma(C)\otimes A$ for $C$ any differential graded algebra, one obtains natural isomorphisms
\begin{eqnarray*}
 \Hom_{\Linfty}(\g\otimes A, \Sigma(\Uinfty(\g))\otimes A)  &\cong &  \Hom_{\Linfty}(\g\otimes A, \Sigma(\Uinfty(\g)\otimes A)) \\
&=& \Hom_{\dgcc_a}(\CE(\g\otimes A),\CE(\Sigma(\Uinfty(\g)\otimes A))) \\
&=& \Hom_{\dgla}(\ST(\g\otimes A),\Sigma(\Uinfty(\g)\otimes A))\\
&=& \Hom_{\dga}(\Uinfty(\g\otimes A),\Uinfty(\g)\otimes A).
\end{eqnarray*}
We define $\tau$ to be the image of $(\gamma\circ \eta)\otimes \id_A$ under this sequence of natural isomorphisms.

\end{proof}

\begin{definition}
Let $M$ be a smooth manifold and $\g$ an $L_\infty$-algebra.
A connection on $M$ with values in $\g$ is a
degree~$1$ element $\alpha$ in $\g \hat{\otimes} \Omega(M)$.
\end{definition}

\begin{definition}
A connection $\alpha$ on $M$ with values in a filtered $L_\infty$-algebra $\g$ is called
flat if $\alpha \in \mathsf{MC}\big(\g \hat{\otimes} \Omega(M)\big)$.
\end{definition}

\begin{definition}
Suppose that $\alpha$ is a connection on $M$ with values in a filtered $L_{\infty}$-algebra $\mathfrak{g}$.
The holonomy $\holinfty_{\alpha} \in \U_\infty(\g)\hat{\otimes} C^\bullet(M)$ of $\alpha$ is the image of $\alpha$ under the composition
$$
\xymatrix{
\g\hat{\otimes} \Omega(M) \ar[r]^(0.45){(\widehat{\eta\otimes \id})_*} & \ST(\g)\hat{\otimes} \Omega(M) \ar[r]^(0.4){\hat{\iota}} & \hat{\U}(\ST(\g)\otimes \Omega(M)) \ar[r]^(0.7){\hat{\tau}}& \cdots}$$

\vspace{-0.7cm}

$$
\xymatrix{
\cdots \ar[r] & \U(\ST(\g))\hat{\otimes} \Omega(M)
\ar[r]^(0.48){(\widehat{\id\otimes \psi})_*} & \U(\ST(\g))\hat{\otimes} C^{\bullet}(M).
}
$$
By definition, the last space equals $\Uinfty(\g)\hat{\otimes} C^{\bullet}(M)$.
The maps above are as follows:
\begin{itemize}
 \item $\eta$ is the map from $\g$ to its strictification $\ST(\g)$.
 \item $\iota$ is the inclusion of $\g$ into its universal enveloping algebra.
 \item $\tau$ is the map defined in Lemma~\ref{lemma:coefficients}.
 \item $\psi$ is Gugenheim's $\mathsf{A}_\infty$ quasi-isomorphism between $\Omega(M)$ and $C^{\bullet}(M)$.
\end{itemize}
\end{definition}

\begin{proposition}
Suppose that $\alpha$ is a flat connection on $M$ with values in a filtered $L_\infty$-algebra $\mathfrak{g}$.
Then $\holinfty_{\alpha}$ is a Maurer--Cartan element of $\Uinfty(\g)\hat{\otimes} C^{\bullet}(M)$.

\end{proposition}

\begin{proof}
All of the maps involved in the definition of $\holinfty_\alpha$ preserve Maurer--Cartan elements.
\end{proof}

Recall that there is a natural inclusion $\Uinfty(\g)\otimes C^{\bullet}(M) \hookrightarrow \Hom(C_{\bullet}(M),\Uinfty(\g))$
of filtered diffential graded algebras. Completing yields a map
$$ \Uinfty(\g)\hat{\otimes} C^{\bullet}(M) \to \Hom(C_{\bullet}(M),\hatUinfty(\g)),$$
which allows us to view $\holinfty_\alpha$ as a map from $C_\bullet(M)$ to $\hat{\U}_\infty(\g)$.
It is not hard to see that the image of this map lies in the kernel $K$ of
the augmentation map $\hatUinfty(\g)\to \mathbb{R}$. Hence if $\alpha$ is flat, this map corresponds to a twisting cochain on $C_{\bullet}(M)$ with values in $K$; see Appendix~\ref{section:twisting_cochains}.
Such a twisting cochain is equivalent to a morphism of differential graded coalgebras from $C_{\bullet}(M)$ to $\hat{\B}\hatUinfty(\g)$,
where $\hat{\B}\hatUinfty(\g)$ denotes the completed bar complex of $\hatUinfty(\g)$.
Hence a flat connection $\alpha$ on $M$ with values in a filtered $L_\infty$-algebra $\g$ gives rise
to a morphism of differential graded coalgebras $C_{\bullet}(M) \to \hat{\B}\hatUinfty(\g)$.
We have proved our main result:

\begin{theorem}\label{main theorem}

Suppose that $\alpha$ is a flat connection on $M$ with values in a filtered $L_\infty$-algebra  $\mathfrak{g}$. Then there is a natural homomorphism of differential graded coalgebras
\[\hol^{\infty}_\alpha\colon  C_\bullet(M) \rightarrow \hat{\B}\hat {\mathbb{U}}_\infty(\g).\]

\end{theorem}

For a flat connection $\alpha$ with values in a filtered differential graded Lie algebra $\g$, one could
also define the holonomy $\hol_{\alpha}$ as the image of $\alpha$ under
$$
\xymatrix{
\g\otimes \Omega(M) \ar[r]^{\hat{\iota}} & \hat{\U}(\g\otimes \Omega(M)) \ar[r]^{\hat{\tau}}& \U(\g)\hat{\otimes} \Omega(M)
\ar[r]^(0.48){(\widehat{\id\otimes \psi})_*} & \U(\g)\hat{\otimes} C^{\bullet}(M).
}
$$
Hence, if $\alpha$ is flat, one obtains a morphism of differential graded coalgebras
$$ \hol_{\alpha}: C_{\bullet}(M) \to \hat{\B}\hat{\U}(\g).$$

\begin{proposition}
Let $\alpha$ be a flat connection on $M$ with values in a filtered differential graded Lie algebra $\g$.
Then the following diagram is commutative:
$$
\xymatrix{
 C_\bullet(M)  \ar[r]^{\holinfty_\alpha} \ar[rd]_{\hol_\alpha}& \hat{\B} \hat{\mathbb{U}}_\infty (\g) \ar[d]^{\B \hat{\mathbb{U}}(\rho)}\\
 &\hat{\B} \hat{\mathbb{U}}(\g).
}
$$
\end{proposition}

\begin{proof}
\hspace*{3pt}This follows from the fact that the Maurer--Cartan elements
$\holinfty_{\alpha} \in\linebreak \Uinfty(\g)\hat{\otimes}C^{\bullet}(M)$
and $\hol_{\alpha} \in \U(\g) \hat{\otimes} C^{\bullet}(M)$ are related by the map
$\U(\rho)\hat{\otimes}\id$.
To establish this, let $\g$ be an arbitrary filtered differential graded Lie algebra and $\phi: A\to B$ morphism of commutative differential graded algebras.
Then the diagrams
\begin{align*}
\xymatrix{
\g^1 \ar[r]^{\hat{\eta}_*} \ar[rrdd]_{\hat{i}}& \ST(\g) \ar[r]^{\hat{i}} & \Uinfty(\g)\ar[dd]^{\U(\rho)}\\
&&\\
&& \U(\g),
}
\end{align*}
and
\begin{align*}
 \xymatrix{
\Uinfty(\g\otimes A) \ar[r]^{\tau} \ar[d]^{\U(\rho)}& \Uinfty(\g)\otimes A \ar[r]^{\id\otimes \phi} \ar[d]^{\U(\rho)\otimes \id}& \Uinfty(\g)\otimes B \ar[d]^{\U(\rho)\otimes \id}\\
\U(\g\otimes A) \ar[r]^{\tau} & \U(\g)\otimes A \ar[r]^{\id\otimes \phi}& \U(\g)\otimes B.
}
\end{align*}
are commutative.
\end{proof}

We saw that in the case of differential graded Lie algebras, the two possible notions of holonomy $\hol$ and $\hol^\infty$ are
related by the quasi-isomorphism $\U(\rho)\colon  \U_\infty(\g)\rightarrow
\U(\g)$. The following lemma shows that, futhermore, both
definitions are consistent with the usual notion of holonomy in the case that $\g$ is a Lie algebra.

\begin{lemma}
Let $\alpha$ be a connection on $M$ with values in a filtered Lie algebra $\g$.
Then $\hol_{\alpha} \in \U(\g)\hat{\otimes} C^{\bullet}(M)$ yields the usual parallel transport
of $\alpha$.
\end{lemma}

\begin{proof}
By degree reasons, $\alpha$ is an element of $\g\hat{\otimes} \Omega^1(M)$ and $\hol_{\alpha}$ is an element of $\U(\g) \hat{\otimes} C^1(M)$.
Let $\gamma: [0,1] \to M$ be a path in $M$.
The pullback of $\alpha$ along $\gamma$ gives an element of
$\g\hat{\otimes} \Omega^1([0,1])$, which can be written as
$$\gamma^*\alpha = \sum_{i=1}^{\infty}\xi^i\otimes a_i(t) dt,$$
where $\mathrm{deg}(\xi^{i}) \to \infty$. The degree of an element in a filtered graded vector space $V$
is the integer $k$ such that  the element is contained in $F_kV$, but not in $F_{k+1}V$.

We consider $\hol_{\alpha} \in \U(\g)\hat{\otimes} C^1(M)$ as a map
from $C_1(M)$ to $\hat{\U}(\g)$. By definition, the evaluation of this map on the path $\gamma$ yields
$$ \sum_{k\ge 1} \sum_{i_1\ge 1,\cdots,i_k\ge 1} (\xi^{i_1} \cdots \xi^{i_k}) \left(\int_{1\ge t_1\ge \cdots \ge t_k \ge 0} a_{i_1}(1-t_1)\cdots a_{i_k}(1-t_k) dt_1\cdots dt_k\right).$$
Up to a shift by $1 \in \U(\g)$, this is the unique solution to the ordinary differential equation
$$ H_0 = 1, \qquad \frac{d}{dt}H_t = \left(\sum_{i=1}^{\infty}\xi^{i}\otimes a_i(1-t) \right)\cdot H_t$$
in $\U(\g)\hat{\otimes} \mathcal{C}^{\infty}([0,1])$.
Hence, $\hol_{\alpha}(\gamma)$ encodes parallel transport of $\alpha$ along $\gamma$ (with reversed orientation).
\end{proof}

The holonomies defined in Theorem~\ref{main theorem} satisfy the following naturality conditions:

\begin{lemma}
Suppose that $\alpha$ is a flat connection on $M$ with values in the filtered $L_\infty$-algebra  $\mathfrak{g}$.
\begin{enumerate}
\item If $f\colon  N \rightarrow M$ is a smooth map, then $\hol^{\infty}_{f^*(\alpha)}=\hol^\infty_\alpha \circ f_*$.
\item If $\gamma\colon  \g \rightarrow \mathfrak{h}$ is a filtered morphism, then
$\hol^\infty_{\gamma_*(\alpha)}=\hat{\B} \hat{\U}_\infty(\gamma) \circ \hol^\infty_\alpha .$

\end{enumerate}
\end{lemma}

\begin{proof}
The first claim follows directly from the naturality of Gugenheim's
$A_\infty$-morph\-ism with respect to the pullback along smooth maps. The second claim is clear since the whole construction is functorial with respect to the coefficient system $\g$.
\end{proof}

\section{Flat connections on configuration spaces}

So far, we constructed an extension of Igusa's higher holonomies \cite{I} to the framework of flat connection with values in $L_\infty$-algebras.
In this section, we explain how rational homotopy theory provides a vast amount of such connections.
We then turn to a specific family of examples, the configuration spaces $\Conf_d(n)$ of $n$ points in $\mathbb{R}^d$ ($d \ge 2$). In \cite{K}, Kontsevich
constructed explicit models for these spaces and used them to establish formality of the chains of the little $d$-disks operad.
We consider the corresponding flat connections, extending considerations of
\v{S}evera and Willwacher \cite{SW} to the higher-dimensional
situation.
Finally, we explain how one can use these flat connections to construct representations of the $\infty$-groupoid of $\Conf_d(n)$,
generalizing the holonomy representations of braid groups.

\subsection{Flat connections and rational homotopy theory}\label{subsection:rational_homotopy_theory}

A Sullivan minimal model of a manifold $M$ is a differential graded algebra
$(A_M,d)$ that is homotopy equivalent to $\Omega(M)$ and is isomorphic, as
a graded algebra, to the free graded commutative algebra $\wedge V$  on a graded vector space $V$. For more details on the definition, we refer the reader to \cite{Sullivan,F}.
For simplicity, we will assume that the homogeneous components of $V$ are finite dimensional.
Such a model exists, for instance, if $M$ has vanishing first cohomology and finite Betti numbers.

As was observed in \cite{Getzler}, the information of a Sullivan model can
be encoded by a flat connection on $M$ that takes values in an $L_\infty$-algebra:
Let $\g$ be the graded vector space with $\g^k = (V^{-k+1})^*$; i.e.,\ $\g$ is the desuspension of the graded dual $V^*$ of $V$.
Observe that since $V$ is concentrated in strictly positive degrees, $\g$ is concentrated in non-positive degrees.
Recall that $\Sym(\s \g)$ denotes the symmetric coalgebra on $\s \g$, the suspension of $\g$.
We equip $\g$ with structure maps
$\mu_n: \Sym^n(\s \g) \to \s \g$
of degree $+1$ given by
$$ \Sym(\s \g) = \Sym(V^*) \hookrightarrow (\Sym(V))^* \rightarrow (\us V)^* \cong \s(\s \g).$$
Here the arrow in the middle that goes from $(\Sym(V))^*$ to $(\us V)^*$ is the map dual of the restriction of the differential $d$ of $\wedge V$ to $V$.
The fact that $d$ squares to zero implies that the maps $(\mu_n)_{n \ge 1}$ equip $\s \g$ with the structure of an $L_\infty$-algebra.
The next step is to consider the morphism $\varphi$. Since $\wedge V$ is free as a commutative graded algebra, it suffices to
know its restriction to $V$. If we choose a homogeneous basis $(v_i)_{i\in I}$ of $V$, we obtain an element
$$\alpha_\varphi := \sum_{i} \varphi(v_i)\otimes v_i$$
of $\Omega(M)\otimes V^*$.
We now consider $\alpha$ as an element of $\Omega(M)\otimes \g$. As such,
$\alpha_\varphi$ has degree $+1$, and the fact that $\varphi$ is a morphism
of commutative differential graded algebras implies that $\alpha_\varphi$ is a Maurer--Cartan element of the differential graded
Lie algebra $\Omega(M)\otimes \g$.
It is clear that one can reconstruct the Sullivan model $(\wedge V,d)$ from $\g$ and $\alpha_\varphi$.
To sum up our discussion, we record the following:
\begin{lemma}\label{lemma:Sullivan_models}
Every finite type Sullivan model of a manifold $M$ corresponds in a natural way to a flat connection on 
$M$ with values in an $L_\infty$-algebra.
\end{lemma}

Let $\alpha_\varphi$ be a flat connection on $M$ associated to a Sullivan model $\varphi: \wedge V \to \Omega(M)$.
In order for the holonomy map
$\holinfty_{\alpha_\varphi}$
from Theorem~\ref{main theorem} to be well defined, we need the series which define it to converge. In Theorem~\ref{main theorem}, this is guaranteed by the assumption that
$\g$ is filtered.
For flat connections associated to Sullivan models, we will circumvent this problem by assuming that $M$ is simply connected.
This allows us to assume that $V$ is concentrated in degrees strictly larger than $+1$, which in turn implies that $\g$ is concentrated
in strictly negative degrees. Consequently, the components of $\alpha_\varphi$ of form degree $0$ and $1$ are zero and no divergent
sums appear in the definition of the holonomy map $\holinfty_{\alpha_\varphi}$.

\begin{theorem}
Let $\varphi\colon \wedge V\to \Omega(M)$ be a Sullivan model of a manifold $M$, and assume that $\wedge V$ is of finite type and $V^1=0$.
Then the holonomy map associated to the flat connection $\alpha_\varphi$ on $M$ with values in $\g$ yields a morphism of differential graded coalgebras
$\holinfty_{\alpha_\varphi}: C_\bullet(M) \to \B\Uinfty(\g). $
\end{theorem}

\begin{remark}
If one composes $\holinfty_{\alpha_\varphi}$ with the projection map
$\B\Uinfty(\g) \cong \B\Omega \CE(\g) \to \CE(\g), $
one obtains essentially the dual to
$$
\xymatrix{
\wedge V \ar[r]^\varphi & \Omega(M) \ar[r]^{\int} & C^{\bullet}(M),
}
$$
where the last map is the usual integration map.
Hence, under mild assumptions (e.g.,\ compactness of $M$) the holonomy map $\holinfty_{\alpha_\varphi}$ will be a quasi-isomorphism of differential graded coalgebras.
Since the adjunction morphism
$ \CE(\g) \to \B\Omega \CE(\g)$
is a quasi-isomorphism, we obtain that $\CE(\g)$ and $C_\bullet(M)$ are quasi-isomorphic dg coalgebras.
Notice that the Baues--Lemaire conjecture \cite{BL}, which was proven by Majewski \cite{Majewski,Majewski_book},
asserts that the strictification $\ST(\g)$ of $\g$ is quasi-isomorphic to Quillen's Lie algebra model $L_M$ of $M$ (\cite{Quillen}).
Hence $\CE(\g)$ is quasi-isomorphic---as a differential graded coalgebra---to Quillen's coalgebra model $\CE(L_X)$.

\end{remark}

\subsection{Flat connections on configuration spaces}\label{subsection:configuration_spaces}

We now turn to a family of specific examples, the configuration spaces of $n$ (numbered) points in $\mathbb{R}^d$, i.e.,
$$ \Conf_d(n):= \{x_1,\dots,x_n \in \mathbb{R}^d: \quad x_i\neq x_j \textrm{ for } i\neq j\}.$$

It turns out to be convenient to consider a natural compactification of
$\Conf_d(n)$ to a semi-algebraic manifold with corners, the Fulton--MacPherson space $\FM_d(n)$.
To obtain these compactifications, one first mods out the action of $\mathbb{R}^d \rtimes \mathbb{R}_{>0}$ by translations and scalings on $\Conf_d(n)$
and then embeds the quotient into
$$ (S^{d-1})^{n \choose 2} \times ([0,\infty])^{n \choose 3}$$
via all relative angles and cross-ratios.
The closure of this embedding naturally admits the structure of a semi-algebraic manifold with corners.
We refer the reader to \cite{Lambrechts,Sinha} for the details of this construction.

The cohomology ring of $\Conf_2(n)$ was determined by Arnold \cite{Arnold} and in higher dimensions by Cohen \cite{Cohen}:
$H^*(\Conf_d(n))$ is the graded commutative algebra
with a set of generators $(\omega_{ij})_{1\le i\neq j \le n}$ of degree $(d-1)$ and the following relations:
$$\omega_{ij} = (-1)^{d}\omega_{ji}, \quad \omega_{ij}\omega_{jk} + \omega_{jk}\omega_{ki} + \omega_{ki}\omega_{ij} = 0.$$

\subsubsection{Kontsevich's models for configuration spaces}

In \cite{K}, Kontsevich constructed a family of graph complexes $^*\Graphs_d(n)$, together with integration maps
$$I:\, ^*\Graphs_d(n) \to \Omega(\FM_d(n)),$$
which are quasi-isomorphisms of commutative differential graded algebras. In dimension $d>2$,
the commutative differential graded algebras $^*\Graphs_d(n)$, together with the integration map $I$, define Sullivan models for $\FM_d(n)$.\footnote{For dimension equal to $2$, the problem is that $^*\Graphs_2(n)$
is not concentrated in positive degrees. Moreover, $I$ does not take values
in smooth differential forms, but in piecewise semialgebraic forms; see \cite{real_homotopy_type} and \cite{Lambrechts} for the technical details.}

Let us recall the definition of $^*\Graphs_d(n)$, following \cite{K} and \cite{Lambrechts}:

\begin{definition}
An admissible graph with parameters $(n,m,k)$, where $ n \geq 1, m \geq 0$, is a finite graph $\Gamma$ such that:
\begin{enumerate}
\item $\Gamma$ has no simple loops.
\item $\Gamma$ contains $n$ \emph{external} vertices, numbered from $1$ to $n$, and $m$ \emph{internal} vertices numbered from $1$ to $m$.
\item $\Gamma$ contains $k$ edges, numbered from $1$ to $k$.
\item Any vertex in $\Gamma$ can be connected by a path to an external vertex.
\item All internal vertices have valency at least $3$.
\item The edges of $\Gamma$ are oriented.
\end{enumerate}
For $n=0$, there is just one graph with parameters $(0,0,0)$, the empty graph ${\emptyset}$.
\end{definition}

\begin{definition}
For every $n\geq 0$ and $d\geq 2$ define $^*\Graphs_d(n)$ to be the
$\mathbb{Z}$-graded vector space over $\mathbb{R}$ generated by equivalence
classes of isomorphism classes of admissible graphs with parameters
$(m,n,k)$. The equivalence relation is generated by the following three conditions:
\begin{itemize}
\item $\Gamma \equiv (-1)^{(d-1)} \Gamma'$, 
if $\Gamma$ differs from $\Gamma'$ by a transposition in the labelling of the edges.
\item $\Gamma \equiv (-1)^{d} \Gamma'$, 
if $\Gamma$ differs from $\Gamma'$ by a transposition in the numbering of the internal vertices.
\item  $\Gamma \equiv (-1)^{d} \Gamma'$, 
if $\Gamma'$ is obtained from $\Gamma$ by reversing the orientation of one of the edges.
\end{itemize}
We define the degree of a  class $[\Gamma]$ with parameters $(n,m,k)$ to be
\[|[\Gamma]|:=(d-1)k-dm.\]
Thus $^*\Graphs_d(n)$ is the direct sum of homogenous components:
\[^*\Graphs_d(n)=\bigoplus_{i\in \mathbb{Z}}\, ^*\Graphs_{d}(n)^i.\]
\end{definition}

\begin{remark}
In view of the equivalence relation, we may assume that, for even $d$,
graphs have no multiple edges, are unoriented, and internal vertices are not ordered. Similarly, for odd $d$, one may assume that the edges are not ordered.
\end{remark}

\begin{definition}
The graded vector spaces $^*\Graphs_d(n)$ have a natural structure of commutative dg algebras. 
The product $\Gamma_1\bullet \Gamma_2$ of $\Gamma_1$ and $\Gamma_2$ is their disjoint union, with the corresponding
external vertices identified. The order in the edges is such that the order of each of the graphs is preserved and $e_1<e_2$ if
$e_i $ belongs to $\Gamma_i$.  Similarly, the numbering of the internal vertices is characterized by the fact that the order in each of
the graphs is preserved and vertices in $\Gamma_1$ have labels smaller than those in $\Gamma_2$. 
The differential $\partial$ is given by the sum over all graphs obtained by contracting one of the edges.
For more precise details on the sings of the differential, please see \cite{Lambrechts}.

\end{definition}

\begin{proposition}[\cite{K, Lambrechts}]
The operations $\bullet$ and $\partial$ give $^*\Graphs_d(n)$ the structure of a commutative differential graded algebra.
\end{proposition}

To a graph $\Gamma$  in $^*\Graphs_d(n)$, one can associate a differential form $\omega_\Gamma \in \Omega(\FM_d(n))$ given by the formula
\[\omega_\Gamma:=\pi_* \Big( \bigwedge_{e \textrm{ edge of } \Gamma}(\pi_e)^* \mathsf{Vol}_{d-1} \Big), \qquad \text{where:}\]

\begin{itemize}
 \item The map $\pi\colon  \FM_d(n+m)\rightarrow \FM_d(n)$
is the natural projection that forgets the last $m$ points on the configuration space.
\item For each edge $e$ of $\Gamma$, $\pi_e\colon  \FM_d(n+m)\rightarrow \FM_d(2)=S^{d-1}$
is the map that sends a configuration of $m+n$ points to the two points that are joined by $e$.
\item $\mathsf{Vol}_{d-1}$ is the rotation invariant volume form of the
$(d-1)$-dimensional sphere, normalized so that its volume is~1.
\end{itemize}

\begin{theorem}[\cite{K,Lambrechts}]\label{proposition_above}
The formula 
$\Gamma \mapsto \omega_\Gamma$
defines a quasi-isomorphism of differential graded algebras:
$I\colon  ^*\Graphs_d(n)\rightarrow \Omega(\FM_d(n))$.
\end{theorem}

\subsubsection{The \v{S}evera--Willwacher connections}

We next introduce flat connection on the compactified configuration spaces $\FM_d(n)$.
In the case $d=2$, these connections where introduced by \v{S}evera and Willwacher \cite{SW}.

\begin{definition}
We say that an admissible graph $\Gamma$ is internally connected if it is non-empty and connected after all the external vertices are removed. We denote by $\CG_d(n)$ the graded vector space spanned by equivalence classes of internally connected graphs with $n$ external vertices, and introduce a grading  by $\overline{\Gamma} :=1+dm-(d-1)k$.
\end{definition}

\begin{remark}
As explained in Subsection~\ref{subsection:rational_homotopy_theory},
Kontsevich's model $^*\Graphs_d(n)$ of the compactified configuration space $\FM_d(n)$
corresponds to a certain flat connection with values in an $L_\infty$-algebra.
Since $^*\Graphs_d(n)$ is the free commutative algebra on the space of internally connected graphs, the graded vector space underlying this $L_\infty$-algebra
is  the space of internally connected graphs $\CG_d(n)$.
The general machinery from
Subsection~\ref{subsection:rational_homotopy_theory} leads to following
definition\slash result:
\end{remark}

\begin{definition}
The \v{S}evera--Willwacher connection $\SW_d(n)$ on $\FM_d(n)$ with values in the $L_\infty$-algebra $\CG_d(n)$ is given by
\[\sum_\Gamma I(\Gamma)\otimes \Gamma \quad \in \quad \Omega(\FM_d(n))\hat{\otimes} \CG_d(n),\]
where the sum runs over a set of graphs whose equivalence classes form a basis of the graded vector space $\CG_d(n)$.
\end{definition}

\begin{proposition}
The \v{S}evera-Willwacher connections are flat.
\end{proposition}

\begin{remark}
We remark that Kontsevich's model $^*\Graphs_d(n)$ for $\FM_d(n)$ is
concentrated in degrees $>1$ and finite-dimensional in each degree if
$d>3$.
However, the $L_\infty$-algebras $\CG_d(n)$ admit filtrations in the sense
of Subsection~\ref{subsection:filtered} for all $d$, and hence our methods are applicable also in the cases $d=2$ and $d=3$. We refer the reader to the forthcoming \cite{AS2} for details.
\end{remark}

\begin{remark}

Applying Theorem~\ref{main theorem} to the flat connections $\SW_d(n)$ yields holonomy maps
$$\holinfty_{\SW_d(n)}: C_\bullet(\FM_d(n)) \to \hat{\B} \Uinfty(\CG_d(n)) \cong \hat{\B} \Omega \CE(\CG_d(n)).$$
The composition of $\holinfty_{\SW_d(n)}$ with the projection to $\CE(\CG_d(n))$ (which is a chain map but not a morphism of coalgebras)
are Kontsevich's formality maps
$$C_\bullet(\FM_d(n)) \to \CE(\CG_d(n))$$
from \cite{K}. Kontsevich proved that these maps are quasi-isomorphisms and that they assemble into a morphism of operads
from $(C_\bullet(\FM_d(n)))_{n\ge 1}$ to $(\CE(\CG_d(n)))_{n\ge 1}$, respectively. It is not hard to verify that the latter operad
of differential graded coalgebras is quasi-isomorphic to its cohomology, which can be identified with the homology operad of the
compactified configuration spaces $(\FM_d(n))_{n\ge 1}$.
This way, Kontsevich established the formality of the chains on the little $d$-disks operad.

The holonomy maps $$\holinfty_{\SW_d(n)}: C_\bullet(\FM_d(n)) \to \hat{\B} \Uinfty(\CG_d(n)) \cong \hat{\B} \Omega \CE(\CG_d(n))$$ that we constructed are extensions of Kontsevich's formality map to a collection of quasi-isomorphisms of differential graded coalgebras.
Therefore, it should be possible to use them to obtain a formality proof that is compatible with the comultiplication
on chains. We hope to report on this in the forthcoming \cite{AS2}.
\end{remark}

\subsection{Drinfeld--Kohno construction in higher dimensions}

If $\g$ is a complex semisimple Lie algebra and $V$ a representation of $\g$, then the braid group
$B_n$ acts on $V^{\otimes n}$. This action comes from the following construction due to Drinfeld and Kohno:
 For each $n\geq 2$ there is a Lie algebra $\td_2(n)$, called the
 Drinfeld--Kohno Lie algebra, and natural flat connections on the configuration spaces $\Conf_2(n)$ with values in $\td_2(n)$, the Knizhnik--Zamolodchikov connections \cite{KZ}.
The Lie algebras $\td_d(n)$ have the property that for any quadratic Lie algebra $\g$ and any representation $V$ of $\g$, there is a morphism of Lie algebras:
$\varphi_n\colon  \td_d(n)\rightarrow \End(V^{\otimes n})$.
Pushing the flat connections along the morphism $\varphi_n$, one obtains flat connections on the vector bundle $V^{\otimes n}$. The holonomy of the flat connection gives an action of the fundamental group of $\Conf_d(n)$, which is the pure braid group $P_n$. Since the connection is compatible with the action of the symmetric group, these actions extend to an action of the braid group $B_n$.
We now explain how this construction can be generalized to higher dimensions. Our aim is to show how the compactified configuration spaces $\FM_d(n)$ act via higher holonomies on the category of representations of quadratic graded Lie algebras.

\begin{definition}
For each dimension $d\geq 2$ and each $n \geq 2$, the Drinfeld--Kohno Lie algebra is the graded Lie algebra $\td_d(n)$ generated by the symbols
$t_{ij}=(-1)^dt_{ji}$  for $1\leq i, j \leq n, i\neq j$, of degree $2-d$,
modulo the relations
\begin{eqnarray*}
{[t_{ij}, t_{kl}]} = 0 \quad & \textrm{if} & \quad  \# \{i,j,k,l\}=4,\\
{[t_{ij}, t_{ik} + t_{jk}]}=0  \quad & \textrm{if} & \quad  \# \{i,j,k\}=3.
\end{eqnarray*}

\end{definition}

These graded Lie algebras $\td_d(n)$ are closely related to the
$L_\infty$-algebras of internally connected graphs $\CG_d(n)$, which
were defined in the previous subsection. In fact, $\td_d(n)$ is just the cohomology of $\CG_d(n)$:

\begin{proposition}[Proposition 6 from \cite{W}]\label{cohomologyCG}
The map $\phi\colon  \td_d(n) \rightarrow H(\CG_d(n)),$\linebreak
defined by sending $t_{ij}$ to the cohomology class of the graph that has only one edge going from the $i$th to the $j$th external vertices, is an isomorphism of graded Lie algebras.
\end{proposition}

The relation between $\td_d(n)$ and $\CG_d(n)$ is even stronger \cite{AS2}:

\begin{proposition}[\cite{AS2}]
The $L_\infty$-algebras $\CG_d(n)$ are formal; i.e.,\ there is an $L_\infty$ quasi-isomorphism between $\CG_d(n)$ and its
cohomology $\td_d(n)$.
\end{proposition}

\begin{remark}

Because the $L_\infty$-algebras $\CG_d(n)$ are formal, one can use
homological perturbation theory to push forward the \v{S}evera--Willwacher
connections $\SW_d(n)$ to flat connection $\widehat{\SW}_d(n)$ with values in the graded Lie algebras $\td_d(n)$.
These induced connections are unique up to gauge equivalence. 
\v{S}evera and Willwacher showed in \cite{SW} that in two dimensions one
recovers the Alsekseev--Torossian connections, which were introduced in \cite{AT}.
\end{remark}

We now show that the graded Lie algebras $\td_d(n)$ naturally act on representations of (a graded version of) quadratic Lie algebras.

\begin{definition}
A quadratic differential graded Lie algebra of degree $D$ is a finite-dimensional differential graded Lie algebra $\g$ together
with a non-degenerate graded symmetric bilinear form
$ \kappa\colon  \g \otimes \g  \rightarrow \mathbb{R}[D],$
satisfying 
 \[ \kappa ([\alpha, \beta], \gamma)= -(-1)^{|\alpha||\beta|} \kappa(\beta, [\alpha,\gamma]) \quad \textrm{and} \quad
\kappa (d \alpha, \beta)= (-1)^{|\alpha|} \kappa(\alpha,d\beta).\]
 
\end{definition}

\begin{example}
 \hspace{0cm}
\begin{enumerate}
 \item A complex semisimple Lie algebra $\g$ endowed with the Killing form is a quadratic differential graded Lie algebra of degree $0$.
\item Let $\g$ be a quadratic Lie algebra and $M$ a closed oriented manifold of dimension $D$. Let $H^-(M)$ denote the graded algebra $ H^-(M)^k:= H^{-k}(M),$
and consider the pairing 
\[ \mu\colon  H^-(M) \otimes H^-(M)\rightarrow \mathbb{R}[D],\]
induced by the Poincare pairing in cohomology. Then the vector space $\g \otimes H^-(M)$ is a quadratic differential graded Lie algebra with bracket
\[ [\alpha \otimes \eta , \beta \otimes \omega]:= (-1)^{|\eta||\beta|}[\alpha, \beta]\otimes \eta \omega\]
and bilinear pairing
$\kappa(\alpha \otimes \eta , \beta \otimes \omega):= (-1)^{|\eta||\beta|} \kappa(\alpha,\beta)\mu(\eta, \omega). $

\end{enumerate}
\end{example}

Let us now fix a quadratic differential graded Lie algebra $\g$ of degree $D$.
We denote by $\U(\g)$ the universal enveloping algebra of $\g$. The bilinear form $\kappa$ defines an 
isomorphism
$\kappa^{\sharp}\colon \g \rightarrow \g^*[D]$,
which induces identifications
$\g \otimes \g[-D] \cong \g \otimes \g^* \cong \End(\g)$.

We will denote by $\Omega$ the element of $(\g \otimes \g) ^{-D}\subset \U
\otimes \U$ that corresponds to $\id \in \End(\g)$ under the identification above. Explicitly, one can choose a basis $(I_\mu)$ for $\g$, with the property that each of the 
basis elements is homogeneous and the basis of  $ \g^*[D] $ induced by the isomorphism $\kappa^\sharp$ is dual to the basis $(I_\mu)$. Then
$\Omega$ can be written as $\Omega= \sum_\mu I_\mu \otimes \tilde{I}_\mu,$
where $\tilde{I}_\mu$ is the unique basis element in $\g ^{|I_\mu|-D}$ with the property that $\kappa(I_\mu, \tilde{I}_\mu)=1$.
In case $D=4l$, there is a potential problem since the bilinear form restricted to $g^{\frac{D}{2}}$ may not be positive definite. In this case, some of the elements $\tilde{I}_\mu$ may not be basis elements but negative of basis elements instead.

The Casimir element of $\g$, denoted by $C$, is the image of $\Omega\in \g \otimes \g $ in the universal enveloping algebra.  Since the bilinear form $\kappa$ is $\mathsf{ad}$ invariant i.e.,\ 
\[\kappa(\mathsf{ad}(x)(y),z)+(-1)^{|x||y|} \kappa(y, \mathsf{ad}(x)(z))=0,\]
the map $\kappa^\sharp\colon  \g \rightarrow \g^*[D]$ is a morphism of representations of $\g$. Since $\id \in \End(\g)$ is an invariant element for the action of $\g$, so is $\Omega$. We conclude that $C$ is a central element of $\U(\g)$.
Also, the compatibility between the differential and the pairing in $\g$
implies that  $\kappa^\sharp\colon  \g \rightarrow \g^*[D]$ is a morphism of chain complexes. Since $\id \in \End(\g)$ is closed, we conclude that $d \Omega=0$.

Recall that $\U(\g)$ admits a coproduct $\Delta\colon  \U(\g) \rightarrow \U(\g) \otimes \U(\g)$,
which is the unique algebra homomorphism with the property that $\Delta(x)= 1 \otimes x + x \otimes 1$
for all $x\in \g$. 

The proof of the following lemma is immediate.
\begin{lemma}\label{lemmaomega}
We regard $\g$ as a subspace of $\U(\g)$ via the obvious inclusion. Then
\[\Omega=\frac{1}{2}(\Delta(C)-1 \otimes C - C\otimes 1).\]
\end{lemma}

Let $ \iota^{12}\colon  \U(\g) \otimes \U(\g) \rightarrow \U(\g) \otimes \U(\g)
\otimes \U(\g)$ be the map $x\otimes y \mapsto x \otimes y \otimes 1$,
and define $\iota^{23}, \iota^{13}$ analogously. Then, for $1\leq i <j
\leq3$, we set $\Omega^{ij}:= \iota^{ij} (\Omega)$.

\begin{lemma}\label{lemmadrinfeld}
The following relation is satisfied: $\quad [\Omega^{12},\Omega^{23}+ \Omega^{13}]=0$.
\end{lemma}
\begin{proof}
First, we observe that since $C$ is a central element in $\U(\g)$, $1 \otimes 1 \otimes C, 1 \otimes C \otimes 1, C \otimes 1 \otimes 1$ are central elements in $\U(\g) \otimes \U(\g) \otimes \U(\g)$. In view of Lemma~\ref{lemmaomega}, we know that for each pair $1 \leq i <j \leq 3$:
$ \Omega^{ij}=\frac{1}{2}\iota^{ij}(\Delta(C))+X^{ij},$
where $X^{ij}$ is central. Therefore, it suffices to prove that
\[[\iota^{12}(\Delta(C)),\iota^{23}(\Delta(C))+ \iota^{13}(\Delta(C))]=0.\]
In order to prove this, we compute
\begin{eqnarray*}
\iota^{23}(\Delta(C))&=&\iota^{23}(1 \otimes C + C \otimes 1 +2 \sum_\mu I_\mu \otimes \tilde{I}_\mu)\\
&=& 1\otimes 1 \otimes C +1\otimes C \otimes 1 +2 \sum_\mu1\otimes I_\mu \otimes \tilde{I}_\mu,
\end{eqnarray*}
and similarly
\[\iota^{13}(\Delta(C))=1\otimes 1 \otimes C + C\otimes 1 \otimes 1 +2 \sum_\mu I_\mu \otimes 1\otimes \tilde{I}_\mu.\]
Therefore, we obtain
\begin{eqnarray*}
\iota^{13}(\Delta(C))+\iota^{23}(\Delta(C))&=& 2  \sum_\mu \Delta(I_\mu)\otimes \tilde{I}_\mu + X,
\end{eqnarray*}
with $X$ central. Finally, we compute
\begin{eqnarray*}
\frac{1}{2}[\iota^{12}(\Delta(C)),\iota^{23}(\Delta(C))+ \iota^{13}(\Delta(C))]&=&[\Delta(C)\otimes 1,\sum_\mu \Delta(I_\mu)\otimes \tilde{I}_\mu]\\
&=&\sum_\mu [\Delta(C), \Delta (I_\mu )]\otimes \tilde{I}_\mu = 0.
\end{eqnarray*}
\end{proof}

\begin{lemma}\label{lemmakey}
Let $\mathfrak{g}$ be a quadratic differential graded Lie algebra of degree $D=d-2$.
For each $n \geq 2$ there is a homomorphism of  graded algebras \[\hat{\varphi}_n:
\U(\td_{d}(n)) \rightarrow \U(\g)^{\otimes n},\] given by the formula $t_{ij}\mapsto  \lambda^{ij}(\Omega)\in \U(\g)^{\otimes n}$,
where $ \lambda^{ij}\colon  \U(\g) \otimes \U(\g) \rightarrow \U^{\otimes n}(\g)$ is the morphism of algebras given by:
\[x \otimes y \mapsto 1 \otimes \dots \otimes 1 \otimes \underbrace{x}_i \otimes 1 \otimes \dots \otimes 1 \otimes \underbrace{ y}_j \otimes 1 \otimes \dots \otimes 1.\]
\end{lemma}
\begin{proof}
We need to prove that  $\hat{\varphi}_n(t_{ij})$ satisfy the defining relations of $\td_d(n)$. It is clear from the definition that  $[\hat{\varphi}_n(t_{ij}), \hat{\varphi}_n(t_{kl})]=0$
if $ \#\{i,j,k,l\}=4$. It remains to prove that
$ [\hat{\varphi}_n(t_{ij}), \hat{\varphi}_n(t_{ik})+ \hat{\varphi}_n(t_{ik})]=0$
if $\#\{ i,j,k\}=3$. Clearly, it is enough to consider the case $n=3$. Thus, it suffices to prove that
$[\Omega^{12},\Omega^{23}+ \Omega^{13}]$
vanishes,
which is precisely the claim of  Lemma~\ref{lemmadrinfeld}.
Since $d\Omega=0$, we conclude that the map $\hat{\varphi}$ is a chain map.
\end{proof}

\begin{corollary}\label{corollaryrep}
Let $\g$ be a quadratic differential graded Lie algebra of degree $D=d-2$ and $V_1, \dots , V_n$ be representations of $\g$.
Then there is a natural homomorphism of graded Lie algebras:
$\varphi\colon  \td_d(n)\rightarrow \End(V_1 \otimes \dots \otimes V_n)$.
\end{corollary}
\begin{proof}
Consider the composition
\[ \U(\td_d(n)) \rightarrow \U(\g)^{\otimes n} \rightarrow \End(V_1 ) \otimes \dots \otimes \End(V_n) \cong \End(V_1\otimes \dots \otimes V_n),\]
where the first map is $\hat{\varphi}_n$ and the second map is the tensor product of the representations.
This is an algebra map that, by the universal property of the enveloping
algebra, corresponds to a morphism of Lie algebras $\varphi\colon  \td_d(n)\rightarrow \End(V_1 \otimes \dots \otimes V_n)$.
\end{proof}

Let $\g$ be a quadratic differential graded Lie algebra of degree $D=d-2$
and $V_1, \dots , V_n$ be finite-dimensional representations of $\g$. By Corollary~\ref{corollaryrep}, there is a morphism of Lie algebras:
\[ \varphi_n\colon  \td_d(n)\rightarrow \End(V_1\otimes \dots \otimes V_n).\]

Recall that pushing forward the \v{S}evera--Willwacher connection $\SW_d(n)$ to cohomology 
results in a flat connection $\widehat{\SW}_d(n)$ on $\FM_d(n)$ with values in $\td_d(n)$.
Pushing forward further along the map $\varphi_n$ then yields a flat connection
$\varphi_n(\widehat{SW}_d(n))$ on the space $\FM_d(n)$ with values in  $\End(V_1\otimes \dots \otimes V_n)$.
Thus, in this way, one obtains flat connections on the trivial graded vector bundle with fiber $V_1 \otimes \dots \otimes V_n$. 
\begin{corollary}
The holonomies of the connections $\varphi_n(\widehat{SW}_d(n))$  give an action of the $\infty$-groupoid\footnote{We adopt the convention that an $\infty$-groupoid is a Kan simplicial set. The $\infty$-groupoid of a space $X$ is the Kan simplicial set of chains $C_\bullet(X)$.} of the space $\FM_d(n)$ on the vector space $V_1 \otimes \dots \otimes V_n$. 
\end{corollary}
In the two-dimensional case, this action corresponds to the usual representations of braid groups on products of representations of quadratic Lie algebras. In future work, we plan to generalize this construction to cyclic $L_\infty$-algebras.
In fact, it seems plausible that this can be achieved directly on the level
of the \v{S}evera--Willwacher connections, which
would allow one to bypass the use of homological perturbation theory. Moreover, we expect the resulting construction to be closely related to Kontsevich's
characteristic classes of cyclic $L_\infty$-algebras from \cite{Feynman}.

\appendix

\section{The completed bar complex}\label{section:twisting_cochains}

\begin{remark}
Let $C$ be a differential graded coalgebra and $A$ a differential graded algebra.
Consider the space $\Hom(C,A)$ of graded linear maps from $C$ to $A$.
The differentials on $C$ and $A$ induce a differential on $\Hom(C,A)$.
Let us denote the comultiplication on $C$ by $\Delta$ and the multiplication on $A$ by $m$.
The convolution product
$$ (f,g) \mapsto m \circ (f\otimes g) \circ \Delta$$
makes
$\Hom(C,A)$ into a differential graded algebra.
\end{remark}

\begin{definition}
For $C$ and $A$ as above, a Maurer--Cartan element\footnote{For the definition of Maurer--Cartan elements of a differential graded algebra, see Remark~\ref{remark:MC_dga}.} of the differential graded algebra
$\Hom(C,A)$ is called a twisting cochain on $C$ with values in $A$.
We denote the set of twisting cochains on $C$ with values in $A$ by $\mathsf{MC}(\Hom(C,A))$.
\end{definition}

\begin{remark}
From now on, we assume that $A$ is an augmented differential graded algebra.
Under certain conditions, a twisting cochain $f$ on $C$ with values in the augmentation ideal $\underline{A}$ is equivalent to a morphism of differential graded coalgebras $F: C \to \B A$.
In fact, given a morphism $F$, one simply obtains a twisting cochain by composition with the projection $\B A \to \underline{A}$.
On the other hand, if we start with a twisting cochain, the natural candidate $\hat{f}$ for the morphism from $C$ to $\B A$
is
$$ \hat{f}(c) = \epsilon(c) + f(c) + f^{\otimes 2}(\Delta c) + f^{\otimes 3}(\Delta^2 c) + \cdots.$$
Here $\epsilon$ is the co-unit of $C$ and $\Delta^n: C\to C^{\otimes n+1}$ is defined iteratively by $\Delta^1=\Delta$ and $\Delta^{n+1} = (\id^{\otimes n}\otimes \Delta)\circ \Delta^n$.
It is not hard to check that, up to convergence issues, $\hat{f}$ is indeed a morphism of differential graded
coalgebras (for instance, if $\Delta^{N}v=0$ for sufficiently large $N$).
One can make this correspondence precise with the help of the completed bar complex.
\end{remark}

\begin{definition}
Let $(C,d,\Delta,\epsilon)$ be a differential graded coalgebra.
A filtration on $C$ is a decreasing sequence of subspaces
$$ F_0(C)=C \supseteq F_1(C)=\mathrm{ker}(\epsilon) \supseteq F_2(C) \supseteq F_3(C) \supseteq \dots,$$
such that:
\begin{enumerate}
\item $ \bigcap_{k} F_k(C) = 0$.
\item If $c \in F_k(C)$, $d(c)$ lies in $F_{k-1}(C)$, and $\Delta(c)$ lies in the linear span of the subspaces $F_i(C)\otimes F_{k-i}(C) \subset C\otimes C$, $i=0,\dots,k$.
\end{enumerate}

A filtered differential graded coalgebra is an differential graded coalgebra
with a filtration.
\end{definition}

\begin{remark}
Let $C$ be a filtered differential graded coalgebra.
We denote the completion of $C$ by $\hat{C}$.
If we equip $C\otimes C$ with the filtration $F_k(C\otimes C)$ given by
the linear span of the subspaces $F_i(C)\otimes F_{k-i}(C)$ for $i=0,\dots,k$,
then the comultiplication becomes continuous. Hence we obtain a map
$$ \hat{\Delta}: \hat{C} \to C\hat{\otimes}C.$$
The differential $d$ and the augmentation map $\epsilon$ extend to $\hat{C}$.
We denote these extensions by $\hat{d}$ and $\hat{\epsilon}$.

\end{remark}

\begin{definition}
In the above situation, we refer to $(\hat{C},\hat{d},\hat{\Delta},\hat{\epsilon})$ as the completion of $(C,d,\Delta,\epsilon)$.
\end{definition}

\begin{remark}
Although the completion $(\hat{C},\hat{d},\hat{\Delta},\hat{\epsilon})$ is \emph{not} a differential graded coalgebra,
since $\hat{\Delta}$ maps into $C\hat{\otimes} C$ instead of $\hat{C}\otimes \hat{C}$,
we will accept the axioms that $(\hat{C},\hat{d},\hat{\Delta},\hat{\epsilon})$ satisfies to be the adequate replacement of what a differential graded coalgebra should be in the completed context.
\end{remark}

\begin{definition}
Let $C$ and $C'$ be two filtered differential graded coalgebras.
A morphism of differential graded coalgebras from $C$ to the completion $\hat{C'}$ of $C'$ is a morphism of chain complexes
$ \varphi: C\to \hat{C'}$
that maps $F_kC$ into $F_k\hat{C'}$ and makes the following two diagrams commute:
\begin{enumerate}
\item
$
\xymatrix{
C \ar[rd]_{\epsilon} \ar[rr]^{\phi} && \hat{C'} \ar[ld]^{\hat{\epsilon'}}\\
& \mathbb{R}&
}
$
\item
$
\xymatrix{
C \ar[dd]^{\Delta} \ar[rr]^{\phi}&& \hat{C'} \ar[dd]_{\hat{\Delta}'} \\
&&  \\
C\otimes C \ar[r]^{\phi\otimes \phi} & \hat{C'}\otimes \hat{C'} \ar[r]& C'\hat{\otimes} C'.
}
$
\end{enumerate}
\end{definition}

\begin{lemma}
If $A$ is an augmented differential graded algebra, then its bar complex $\B A$ is naturally a
filtered differential graded coalgebra.
The filtration on $\B A$ is defined by $$F_k(\B A):= \bigoplus_{l \ge k} \T^{l}(\s \underline{A}).$$
Recall that $\underline{A}$ denotes the kernel of the augmentation map.
\end{lemma}

\begin{definition}
Let $A$ be an augmented differential graded algebra.
The completed bar complex of $A$ is the completion of the filtered
differential graded coalgebra $\B A$. We denote it by $\hat{\B} A$.
\end{definition}

\begin{proposition}
Let $C$ be a differential graded coalgebra and $A$ an augmented
differential graded algebra.
The following maps define a bijection between the set $\mathsf{MC}(\Hom(C,\underline{A}))$ of twisting cochains
on $C$ with values in $\underline{A}$ and the set of morphisms from $C$ to the completed bar complex $\hat{\B}A$ of $A$:
\begin{enumerate}
\item Given a morphism $F: C \to \hat{\B}A$, the associated morphism $\check{F}: C \to \underline{A}$
is just the composition of $F$ with the projection $\hat{\B} A \to \underline{A}$.
\item Given $f: C \to \underline{A}$ a twisting cochain on $C$ with values in $\underline{A}$, we define $\hat{f}: C \to \hat{\B} A$ to be
    $$\hat{f}(c):= \epsilon(c) + f(c) + f^{\otimes 2}(\Delta c) + f^{\otimes 3}(\Delta^2 c) + \cdots.$$
\end{enumerate}
\end{proposition}

\begin{proof}
Observe first that the map $\hat{f}$ associated to $f: C \to \underline{A}$
is well defined as a map from $C$ to the completed bar complex of $A$.
It is not hard to check that the given assignments map twisting cochains to coalgebra maps and vice versa
and that the twisting cochain associated to the morphism $\hat{f}$ is $f$ itself.

It remains to check that the morphism associated to the twisting cochain $\check{F}$ is $F$ itself.
This is a consequence of the following lemma:
\end{proof}

\begin{lemma}
Let $F$, $G: C\to \hat{\B}A$ be morphisms such that $\check{F} = \check{G}$. Then $F = G$.
\end{lemma}

\begin{proof}
We denote the projection of an element $x \in \hat{\B} A$ to $\T^{k}\underline{A}$ by $x^k$.

By definition, we have
$$ F(c)^{0} = \epsilon(c) = G(c)^0 \qquad \textrm{and} \qquad F(c)^1=\check{F}(c)=\check{G}(c)=G(c)^{1}.$$
Now one can proceed by induction: Let $k\ge 1$ and suppose that $F(c)^{l} = G(c)^{l}$ holds for
all $l\le k$.
Since $\Delta(F(c)) = (F\otimes F)(\Delta c)$ and since $\overline{\Delta}: \overline{\B A} \to \overline{\B A} \otimes \overline{\B A}$ is injective,
we know that $F(c)^{k+1}$ is determined by $F(c)^{l}$ for $l\le k$.
Details can be found in \cite[Appendix B]{Quillen}.
\end{proof}


\thebibliography{10}
\bibitem{AT}
A. Alekseev and C. Torossian, {\em Kontsevich deformation quantization and
flat connections}, Comm. Math. Phys. \textbf{300} (2010), no. 1, 47--64.

\bibitem{Rossi}
J. Alm and C. Rossi,
{\em The universal enveloping algebra of an  $L_\infty$-algebra via deformation quantization}, in preparation.

\bibitem{AS}
C. Arias Abad and F. Sch\"atz,
{\em The $\mathsf{A}_\infty$ de Rham theorem and the integration of representations up to homotopy}, Int. Math. Res. Not. {\bf 16} (2013),
3790--3855.

\bibitem{AS2}
C. Arias Abad and F. Sch\"atz,
{\em Holonomies of the \v{S}evera-Willwacher connection}, in preparation.


\bibitem{Arnold}
V.I. Arnold, {\em The cohomology ring of the group of dyed braids}, Mat. Zametki {\bf 5} (1969), 227--231.

\bibitem{B}
V. Baranovsky, {\em A universal enveloping for $L_\infty$-algebras}, Math. Res. Lett. {\bf 15} (2008).

\bibitem{BL}
H.J. Baues, and J.M. Lemaire, {\em Minimal models in homotopy theory}, Math. Ann. {\bf 225} (1977), 219--242.

\bibitem{BS}
J. Block and A. Smith, {\em The higher Riemann-Hilbert correspondence}, Adv. Math. {\bf 252} (2014), 382--405.


\bibitem{Chen}
K.T. Chen,
{\em Iterated path integrals}, Bull. Amer. Math. Soc. {\bf 83} (1977), 831--879.

\bibitem{Cohen}
F. Cohen, {\em The homology of $C_{n+1}$-spaces, $n\ge 0$}, The homology of iterated loop spaces, Springer-Verlag,
Berlin (1976), Lect. Notes in Math. \textbf{533}, 207--351.

\bibitem{F}
Y. F\'elix, S. Halperin and J-C. Thomas, {\em Rational homotopy theory},
Graduate Texts in Mathematics {\bf 205}, Springer-Verlag, New York, 2001.

\bibitem{Getzler}
E. Getzler, {\em Lie theory for nilpotent $L_\infty$-algebras}, Ann.\ of Math.\ (2) {\bf 170} (2009), 271--301.


\bibitem{G}
V.K.A.M. Gugenheim,
{\em On Chen's iterated integrals},
Illinois J. Math. \textbf{21} (1977), no. 3, 703--715.

\bibitem{real_homotopy_type}
R. Hardt, P. Lambrechts, V. Turchin and I. Voli\'c,
{\em Real homotopy theory of semi-algebraic sets},
Algebr. Geom. Topol. {\bf 11} (2011), 2477--2545.


\bibitem{Hinich}
V. Hinich, {\em DG coalgebras as formal stacks}, J. Pure Appl. Algebra \textbf{162} (2001), 209--250.

\bibitem{I}
K. Igusa,
{\em Iterated integrals of superconnections}, arXiv:0912.0249.

\bibitem{KZ}
V.G. Knizhnik, and A.B. Zamolodchikov, 
{\em Current algebra and Wess Zumino model in two dimensions}, Nucl. Phys.
B \textbf{247} (1984), 83--103.

\bibitem{Feynman}
M. Kontsevich, {\em Feynman diagrams and low-dimensional topology}, in First European Congress of
Mathematics, Vol. II (Paris, 1992), volume 120 of Progr. Math., Birkh\"auser, Basel (1994), 97--121.


\bibitem{K}
M. Kontsevich,
{\em Operads and motives in deformation quantization}, Lett. Math. Phys.
\textbf{48} (1999), 32--72.

\bibitem{LM}
T. Lada and M. Markl, {\em Strongly homotopy Lie algebras}, Comm. Algebra \textbf{23} (1995), 2147--2161.


\bibitem{Lambrechts}
P. Lambrechts and I. Voli\'{c},
{\em Formality of the little $N$-discs operad}, Mem. Amer. Math. Soc. {\bf 230}, no. 1079, forthcoming.


\bibitem{Majewski}
M. Majewski, {\em A proof of the Baues-Lemaire conjecture in rational homotopy theory},
Proceedings of the Winter School ``Geometry and Physics'' (Srn\'i, 1991), Rend. Circ. Mat. Palermo (2) suppl.\ no.\ {\bf 30} (1993), 113--123.


\bibitem{Majewski_book}
M. Majewski, {\em Rational homotopical models and uniqueness}, Mem. Amer. Math. Soc. \textbf{143} (2000), no. 682.

\bibitem{Picken}
J. Faria Martins and R. Picken,
{\em Surface Holonomy for Non-Abelian 2-Bundles via Double Groupoids}, Adv. Math.  \textbf{226} (2011), 3309--3366.

\bibitem{Quillen}
D. Quillen, {\em Rational homotopy theory}, Ann.\ of Math. (2) \textbf{90} (1969), 205--295.

\bibitem{SatiSchreiberStasheff}
H. Sati, U. Schreiber, and J. Stasheff,
{\em $L-\infty$-algebra connections and applications to String- and Chern-Simons $n$-transport}, in Quantum Field Theory, Birkh\"auser (2009), 303--424.

\bibitem{SchreiberWaldorf}
U. Schreiber and K. Waldorf,
{\em Connections on non-abelian Gerbes and their Holonomy}, Theory Appl. Categ. {\bf 28} (2013), 476--540.

\bibitem{SW}
P. \v{S}evera and T. Willwacher, {\em Equivalence of formalities of the
little disks operad},  Duke Math. J. \textbf{160} (2011), 175--206.

\bibitem{Sinha}
D. Sinha, {\em Manifold-theoretic compactifications of configuration spaces},
Selecta Mathematica (N.S.) {\bf 10} (2004), 391--428.

\bibitem{Sullivan}
D. Sullivan, {\em Infinitesimal computations in topology}, Inst. Hautes \'Etudes Sci. Publ. Math. {\bf 47} (1977), 269--331.



\bibitem{T}
D. Tamarkin, {\em Formality of chain operad of little disks}, Lett. Math.
Phys. \textbf{66} (2003), 65--72.



\bibitem{Tradleral}
T. Tradler, S. Wilson, and M. Zeinalian,
{\em Equivariant holonomy for bundles and abelian gerbes}, Comm. Math. Phys. {\bf 315} (2012), 38--108.

\bibitem{W}
T. Willwacher, {\em M. Kontsevich's graph complex and the Grothendieck-Teichm\"uller Lie algebra}, arXiv:1009.1654.


\bibitem{Yekutieli}
A. Yekutieli, {\em Nonabelian Multiplicative Integration on Surfaces},
arXiv: 1007.1250.
\end{document}